\documentclass[10pt]{amsart}

\usepackage{amsfonts}
\usepackage{amscd}
\usepackage{amsmath}
\usepackage{amsthm}
\usepackage{amssymb}
\usepackage{amsxtra}

\input cyracc.def
\DeclareFontFamily{U}{russian}{}
\DeclareFontShape{U}{russian}{m}{n}
        { <5><6> wncyr5
        <7><8><9> wncyr7
        <10><10.95><12><14.4><17.28><20.74><24.88> wncyr10 }{}
\DeclareSymbolFont{Russian}{U}{russian}{m}{n}
\DeclareSymbolFontAlphabet{\mathcyr}{Russian}
\makeatletter
\let\@math@cyr\mathcyr
\renewcommand{\mathcyr}[1]{\@math@cyr{\cyracc #1}}
\makeatother

\input xy
\xyoption{all}

\newcommand{\bL}{\mathbf{L}}
\newcommand{\bR}{\mathbf{R}}

\newcommand{\caC}{{\mathcal C}}

\newcommand{\caO}{{\mathcal O}}

\newcommand{\caL}{{\mathcal L}}

\newcommand{\integers}{\mathbf{Z}}
\newcommand{\naturals}{\mathbf{N}}
\newcommand{\bG}{\mathbb{G}}

\newcommand{\bQ}{\mathbb{Q}}

\newcommand{\Hom}{\mathrm{Hom}}
\newcommand{\Ho}{\mathsf{Ho}}
\newcommand{\uHom}{\underline{\mathrm{Hom}}}
\newcommand{\Ab}{\mathrm{\bf Ab}}
\newcommand{\sAb}{\mathrm{\bf sAb}}
\newcommand{\Set}{\mathrm{\bf Set}}
\newcommand{\sSet}{\mathrm{\bf sSet}}

\newcommand{\colim}{\mathrm{colim}}
\newcommand{\holim}{\mathrm{holim}}
\newcommand{\hocolim}{\mathrm{hocolim}}

\newcommand{\op}{\mathrm{op}}

\newcommand{\Mod}{\mathrm{Mod}}

\newcommand{\Spec}{\mathrm{Spec}}

\newcommand{\LQ}{\mathsf{L}\mathbf{Q}}

\newcommand{\Cpx}{\mathsf{Cpx}}

\theoremstyle{theoremstyle}
\newtheorem{theorem}{Theorem}[section]
\newtheorem*{theorem*}{Theorem}
\newtheorem{lemma}[theorem]{Lemma}
\newtheorem{proposition}[theorem]{Proposition}
\newtheorem*{proposition*}{Proposition}
\newtheorem{corollary}[theorem]{Corollary}
\newtheorem*{corollary*}{Corollary}

\newtheorem*{conjecture*}{Conjecture}

\newtheorem{definition}[theorem]{Definition}
\newtheorem{definition*}{Definition}
\newtheorem{remark}[theorem]{Remark}
\newtheorem{remark*}{Remark}

\newcommand{\A}{\mathbb{A}}
\newcommand{\C}{\mathsf{C}}
\newcommand{\D}{\mathsf{D}}
\newcommand{\Q}{\mathbb{Q}}

\newcommand{\PP}{\mathbf{P}}

\newcommand{\MZ}{\mathsf{M}\mathbf{Z}}

\newcommand{\unit}{\mathbf{1}}

\newcommand{\id}{\mathrm{id}}

\renewcommand{\S}{\mathbb{S}}

\newcommand{\DM}{\mathsf{DM}}

\newcommand{\DMT}{\mathsf{DMT}}
\newcommand{\gm}{\mathrm{gm}}

\newcommand{\cart}{\mathrm{cart}}
\newcommand{\Alg}{\mathrm{Alg}}
\newcommand{\Tot}{\mathrm{Tot}}
\newcommand{\Perf}{\mathrm{Perf}}
\newcommand{\Total}{\mathrm{Total}}

\newcommand{\UC}{\mathrm{UC}}
\newcommand{\ModCat}{\mathsf{ModCat}}
\newcommand{\Sect}{\mathsf{Sect}}
\newcommand{\ev}{\mathrm{ev}}
\newcommand{\Rep}{\mathrm{Rep}}
\newcommand{\perf}{\mathrm{perf}}
\newcommand{\str}{\mathrm{str}}
\newcommand{\im}{\mathrm{im}}

\newcommand{\B}{\mathcyr B}

\title{Derived fundamental groups for Tate motives}
\author{Markus Spitzweck}
\date{\today}
\begin{document}

\pagestyle{plain}
\maketitle

\begin{abstract}
We construct derived fundamental group schemes
for Tate motives over connected smooth schemes over fields.
We show that there exists a pro affine derived group scheme
over the rationals such that its category of perfect representations
models the triangulated category of rational mixed Tate motives.
Under a hypothesis which is
weaker than an integral version of the Beilinson-Soul$\acute{\mathrm{e}}$ vanishing
conjecture we show that there is an affine derived group scheme
over the integers such that its perfect representations model Tate
motives with integral coefficients. The hypothesis is for
example fulfilled for number fields.
This generalizes previous non-derived
constructions of fundamental group schemes
for Tate motives with rational coefficients.
\end{abstract}

\tableofcontents

\section{Introduction}

The main purpose of this paper is to furnish an unconditional construction 
of fundamental group schemes for triangulated rational mixed Tate motives of smooth connected
schemes over fields. To our knowledge this construction is the first one with the property
that the fundamental group models Tate motives over such general base schemes.
In order to achieve this we shall make use
of a derived formalism which enables us to avoid previous
assumptions such as the Beilinson-Soul$\acute{\mathrm{e}}$ vanishing
conjecture or the $K(\pi,1)$-conjecture.

It turns out that in general we have to consider pro-objects in the category
of affine derived group schemes in order for the representation category
to model the Tate motives in a correct way. We will explain this more detailed
later in the introduction.

Under an integral vanishing assumption which is weaker than an integral
version of the Beilinson-Soul$\acute{\mathrm{e}}$ vanishing
conjecture we show that there is an affine derived group scheme over
the {\em integers} modelling Tate motives.

This applies in particular to Tate motives over a number field. The
resulting group scheme over the integers can be thought
of as a natural integral structure on the usual (non-derived) rational Tate motivic fundamental
group of such a field.

Let us recall in which environment our considerations take place.
Triangulated and abelian categories of mixed Tate motives
have been constructed in different ways
\cite{bloch.lie}, \cite{bloch-kriz}, \cite{km}, \cite{deligne-goncharov},
\cite{voevodsky.triangulated}, \cite{cisinski-deglise}.
One possibility is to start with a triangulated category of mixed motives $\DM(S)$
over the base scheme $S$ and to consider the full triangulated
subcategory generated by the Tate motives (see e.g. \cite{roendigs-oestvar.modules}).
In general it is still unclear
which construction of $\DM(S)$ is the correct one,
i.e. for arbitrary base schemes such as $\Spec(\integers)$ with general coefficients such as the
integers. 

In \cite{km} Tate motives with integer coefficients over a field $k$ have
been constructed by considering module categories over Adams graded
(i.e. possessing an additional grading)
$E_\infty$-algebras in the category of chain complexes of abelian groups
which come from cycle complexes with partial
multiplications. It is still not settled if the resulting triangulated
categories fully embed into $\DM(k)$ as the Tate objects. 
In \cite{spitzweck.per} we gave another construction of Adams graded $E_\infty$-algebras
together with an embedding of the resulting module category
into $\DM(k)$.

We switch now to rational coefficients. In this case embeddings
into $\DM(k)_\bQ$ of module categories over rational cycle dga's have
already been constructed in \cite{spitzweck-nistech} and \cite{spitzweck-mot},
see \cite[II.5.5.4, Th. 111, II.5.5.5]{levine.survey-mixed-motives} for a summary.

In the case where the Beilinson-Soul$\acute{\mathrm{e}}$ vanishing
conjecture holds for the field $k$ the rational triangulated category
of mixed Tate motives over $k$ admits a $t$-structure, see \cite{levine.tate-vanishing}.
It is possible to describe the heart of this $t$-structure
as the representations of an affine group scheme over $\bQ$,
see \cite{km}. In case the $K(\pi,1)$-conjecture holds for
the field $k$ the triangulated category of Tate motives is
then the derived category of this abelian representation category
(we note that we do not distinguish between big and small categories
of Tate motives in this introduction).

In this article we want to generalize this picture in two different
ways: First we want to describe unconditionally the full triangulated category as
representation category and second we want to get rid of rational
coefficients. 

For the first point it is necessary in general to pass from group
schemes to so-called derived group schemes and pro derived group schemes. We shall explain
these notions below. It turns out that the same procedure also settles
the second issue if our weak vanishing assumption is fulfilled, see further below.

Since we work intrinsically derived it is also not necessary
that a version of the $K(\pi,1)$-conjecture holds. Moreover
our approach builds on the particular case where a vanishing result holds
which is weaker than the Beilinson-Soul$\acute{\mathrm{e}}$ 
vanishing conjecture. We will state this condition further below
in the introduction. In this special case a description of
Tate motives as representations is possible without passing to
pro objects (on the side of the group schemes) and with integral coefficients.

Let us recall how the group scheme
whose representation category is the heart of the above mentioned $t$-structure
is constructed. 
One starts with an augmented (additionally Adams graded) commutative
dga $A$ over $\Q$ modelling Tate motives. The Bar construction $B(\bQ,A,\bQ)$, using the augmentation,
gives a simplicial commutative Adams graded cdga, and the total complex
inherits the structure of an Adams graded dg Hopf algebra. The zeroth
cohomology of this object is the Hopf algebra representing the (non-derived)
group scheme known from e.g. \cite{km}.

The algebra underlying this total object has the natural
description as a derived tensor product $\bQ \otimes^\bL_A \bQ$,
and in fact it is the first stage in the $\check{\mathrm{C}}$ech resolution of the
augmentation $A \to \bQ$. We refer to \cite{wojtkowiak} for
for the geometric intuition of cosimplicial path spaces underlying
this construction.

The step to produce a truly derived group scheme is to consider
the full $\check{\mathrm{C}}$ech resolution as a cosimplicial algebra, i.e.
the object $[n] \mapsto B^n:=\bQ^{\otimes^\bL_A (n+1)}$. 
A representation of this derived group scheme is then a cosimplicial
$B^\bullet$-module $M^\bullet$ which is homotopy cartesian,
i.e. for any map $\varphi \colon [n] \to [m]$ in $\bigtriangleup$
the natural map $M^n \otimes^\bL_{B^n} B^m \to M^m$ is an equivalence.

This picture has an immediate generalization to integral coefficients.
Then instead of working with strictly commutative dga's one has to
work with $E_\infty$-algebras.

Let us first briefly indicate what an affine derived group scheme is in general
over an $E_\infty$-algebra $A$. It is defined to be a cosimplicial $A$-algebra
$B^\bullet$ which satisfies three conditions:  i) the so-called Segal
conditions, i.e. the natural maps
$(B^1)^{\otimes^\bL_A n} \to B^n$ are equivalences, ii) a condition for
being group-like, iii) the map $A \to B^0$ is
an equivalence.

We will construct an Adams graded affine
derived group scheme in the following way:
we start with an augmentated Adams graded $E_\infty$-algebra $A$ modelling
Tate motives over some given base (with integral coefficients). This
has a natural augmentation $A \to \unit$. We view the map
$\Spec(\unit) \to \Spec(A)$ as a covering of $\Spec(A)$ and build
the cosimplicial algebra $[n] \mapsto B^n := \unit^{\otimes^\bL_A (n+1)}$
which can be considered as the $\check{\mathrm{C}}$ech resolution of this covering. 
The cosimplicial $A$-algebra $B^\bullet$ is then an affine derived
group scheme. As above we can talk about representations of this
group scheme. 

A representation $M^\bullet$ of $B^\bullet$ will be called
perfect if every $M^n$ is a perfect $B^n$-module, or
equivalently if $M^0$ is a perfect $B^0$-module. 
We denote the category of perfect represenations of $B^\bullet$
by $\Perf(B^\bullet)$.

We need an additional notion for $A$ to formulate our intermmediate result.
For an Adams graded object $X$ let us denote by $X(k)$ the part
of $X$ sitting in Adams degree $k$.
Since $A$ models Tate motives it will be of the sort that
$A(k) \simeq 0$ for $k > 0$ and the unit map $\integers \to A(0)$
is an equivalence. We will say that $A$ is of bounded Tate type
if each complex $A(k)$, $k<0$, is cohomologically bounded from below.

Our first main theorem states then that if $A$ is of
bounded Tate type then the category of perfect 
representations of $B^\bullet$ is equivalent to the category
of perfect $A$-modules.

This can be viewed as the statement that the map
$\Spec(\unit) \to \Spec(A)$ is really a covering in the sense
that it satisfies descent for perfect modules.

These statements take place in the world of Adams
graded complexes. By proposition (\ref{j-equiv})
representations of $\bG_m$ are Adams graded chain complexes,
so we can consider the semi-direct product of our affine
derived group scheme $B^\bullet$ with $\bG_m$ to obtain
a group scheme whose perfect representations in chain complexes
gives back the triangulated (or $\infty$-) category of mixed Tate 
motives over the given base. 

In the general case, i.e. where our weakened version of the
Beilinson-Soul$\acute{\mathrm{e}}$ vanishing conjecture
is not satisfied, we have to consider a pro-system of
affine derived group schemes over the rationals which arises in the followig way:
We write an Adams graded cdga $A$
modelling Tate motives over the rationals as the filtered (homotopy)
colimit of algebras $A_i$ which are finitely presented. For each of the $A_i$
we built the affine derived group scheme $B_i^\bullet$ as above.
This defines a pro affine derived group scheme $\text{``$\lim_i$''} B_i^\bullet$.
The category of perfect representations $\Perf(\text{``$\lim_i$''} B_i^\bullet)$
of $\text{``$\lim_i$''} B_i^\bullet$ is defined
to be the colimit of the $\Perf(B_i^\bullet)$.
Our second main theorem states that $\Perf(\text{``$\lim_i$''} B_i^\bullet)$ is equivalent
to the category of geometric mixed Tate motives as tensor
triangulated categories.

We remark that this relationship can be compared to the envisaged
theory of derived Tannakian duality of \cite{toen.habil}.
We are able to write a given category as the category
of perfect representations of a (pro) affine derived group scheme.
We note that contrary to the conditions stated in loc. cit.
we do not need the existence of a $t$-structure.
We note that the statements of loc. cit. are still conjectural.

{\bf Acknowledgements.} I thank Paul-Arne {\O}stv{\ae}r and the university of Oslo
for the hospitality which I enjoyed during the preparation of this text.
I thank Peter Arndt for helpful suggestions and discussions.

\section{Statement of the main results}

Let $k$ be a field, $X$ a smooth connected $k$-scheme
(of finite type). Then the triangulated category $\DM(X)$ of motives over $X$
is defined (see \cite[Definition 10.1.1]{cisinski-deglise}).

We denote by $\DMT(X)$ the full triangulated subcategory of $\DM(X)$ generated by the
$\integers(i)$, $i \in \integers$, and closed under sums. Also let $\DMT_\gm(X)$ be the 
full triangulated subcategory of $\DMT(X)$ generated by the
$\integers(i)$.

We denote by $\DM(X)_R$, $\DMT(X)_R$ and $\DMT_\gm(X)_R$ the versions
with $R$-coefficients for a commutative ring $R$.

Our main theorem for general bases and with rational coefficients
reads as follows:

\begin{theorem} \label{main1-thm}
There is a pro affine derived group scheme
$\text{\rm ``$\lim_i$''} B_i^\bullet$ over $\Q$ (for the definition
of affine derived group scheme see definition (\ref{defi-derived-affine-gr-scheme})
and of pro affine derived group scheme definition (\ref{jhgyr}))
such that the category of perfect representations
$\Perf(\text{\rm ``$\lim_i$''} B_i^\bullet)$ of $\text{\rm ``$\lim_i$''} B_i^\bullet$
(see also section (\ref{hgddf}))
is naturally equivalent to $\DMT_\gm(X)_\Q$ as tensor triangulated category.
For a commutative $\Q$-algebra $R$ let $B_{i,R}^\bullet=B_i^\bullet \otimes_\Q^\bL R$. Then we have
an equivalence of tensor triangulated categories $\DMT_\gm(X)_R \simeq \Perf(\text{\rm ``$\lim_i$''} B_{i,R}^\bullet)$. 
\end{theorem}

We give the proof at the end of section (\ref{transfer}).

We remark that this theorem has an analogue for Beilinson motives, where more general
base schemes (which need not lie over a field) can be considered, see section (\ref{hgfddf}).

Next we turn to our statements for general coefficients.

\begin{theorem} \label{main-thm}
Suppose for each $i>0$ there is an $N \in \integers$ such that
$$\Hom_{\DM(X)}(\integers(0),\integers(i)[n])=0$$ for $n < N$.
Then there is an affine derived group scheme $B^\bullet$ over $\integers$ 
such that the category of perfect representations $\Perf(B^\bullet)$ of $B^\bullet$
is naturally equivalent to $\DMT_\gm(X)$ as tensor triangulated category.

For a commutative ring $R$ let $B_R^\bullet=B^\bullet \otimes^\bL R$. Then we have
an equivalence of tensor triangulated categories $\DMT_\gm(X)_R \simeq \Perf(B_R^\bullet)$. 
\end{theorem}

We give the proof at the end of section (\ref{transfer}).

The affine derived group scheme $B^\bullet$ can therefore be viewed as the {\em derived
motivic fundamental group} of $\DMT_\gm(X)$ (or rather of the infinity categorical
version of $\DMT_\gm(X)$).

We also have the following generalization of theorem (\ref{main-thm}).

\begin{theorem} \label{R-main-thm}
Let $R$ be a commutative ring of finite homological dimension.
Suppose for each $i>0$ there is an $N \in \integers$ such that
$$\Hom_{\DM(X)_R}(R(0),R(i)[n])=0$$ for $n < N$.
Let $R \to R'$ be a map of commutative rings.
Then there is an affine derived group scheme $B^\bullet$ over $R'$
such that the category of perfect representations $\Perf(B^\bullet)$ of $B^\bullet$
is naturally equivalent to $\DMT_\gm(X)_{R'}$ as tensor triangulated category.
\end{theorem}

The proof is also given at the end of section (\ref{transfer}).

We give examples of situations where these theorems apply
in section (\ref{section:examples}).

\section{Preliminaries and notation} \label{prelim}

In this text we deal with the theory of $\S$-modules and algebras in the algebraic setting
as developed in \cite{km}. Here $\S$ denotes the monoid $\caL(1)$, where $\caL$ is the image
of the topological linear isometries operad in $\Cpx(\Ab)$ under the normalized chain complex
functor. Note that in \cite{km} the notation $\mathbb{C}$ is used instead of $\S$.
The category of $\S$-modules is endowed witht the tensor product
$$M \boxtimes N = \caL(2) \otimes_{\S \otimes \S} (M \otimes N).$$
It is shown in \cite{km} that this indeed defines a symmetric monoidal structure
on $\Cpx(\Ab)$ with a so-called pseudo unit. This means that there is a unitality
map $\integers \boxtimes M \to M$ for any $M$ satisfying natural properties,
but this map need not be an isomorphism. We refer to \cite{spitzweck.thesis} for more
background on this material, in particular for the model structures these categories
are equipped with.

What will be important for us is that the category of commutative $\S$-algebras is
equipped with a proper model structure, see \cite{mandell.flat}. A commutative
$\S$-algebra is the same thing as an $\caL$-algebra, see \cite{km}, so the theory
of commutative $\S$-algebras is the same as the theory of $E_\infty$-algebras for
the particular $E_\infty$-operad $\caL$.

We will also speak about $\S$-modules and algebras in categories which are related
to $\Cpx(\Ab)$ such as $\Cpx(\Ab)^\integers$ (Adams graded complexes),
$\Cpx(R)$, the category of chain complexes
of $R$-modules for a commutative ring $R$, $\Cpx(R)^\integers$ and
complexes of modules over a cosimplicical commutative algebra.

For a commutative $\S$-algebra $A$ we write $\D(A)$ for the homotopy category
of $A$-modules. It is a closed symmetric monoidal category.

We write $\Perf(A)$ for the full tensor subcategory of $\D(A)$ of
perfect objects.

We usually write $\otimes^\bL$ for the derived tensor product, for example the tensor
product in $\D(A)$ will be written $\otimes^\bL_A$.
If we have two $\S$-algebra morphisms $A \to B$ and $A \to C$ we denote
by $B \otimes_A C$ the pushout.

We denote by $\bigtriangleup$
the simplicial category, i.e. objects
are the non-empty finite ordered sets $[n]=\{0,\ldots,n\}$
and morphisms are the monotone maps.

If $B^\bullet$ is a cosimplicial commutative $\S$-algebra we write $\D(B^\bullet)$ for the derived
category of $B^\bullet$-modules, where a $B^\bullet$-module consists of $B^m$-modules
$M^m$ together with maps $M^n \to M^m$ for each map $[n] \to [m]$ in
$\bigtriangleup$, linear over the corresponding map $B^n \to B^m$.
This is a particular example of a section category. This material is
covered in section (\ref{diag-mod-cat}).

We write $\D(B^\bullet)_\cart$ for the full subcategory of cartesian objects
of $\D(B^\bullet)$ (sometimes we also say homotopy cartesian),
i.e. for those modules $M^\bullet$ such that for any map $[n] \to [m]$
in $\bigtriangleup$ the induced map $M^n \otimes^\bL_{B^n} B^m \to M^m$ is
an isomorphism in $\D(B^m)$.

We let $\Perf(B^\bullet)$ be the full subcategory of $\D(B^\bullet)_\cart$ consisting
of $B^\bullet$-modules $M^\bullet$ such that each $M^n$ is a perfect $B^n$-modules,
or equivalently such that $M^0$ is a perfect $B^0$-module.

Suppose $A \to B^\bullet$ is a coaugmented cosimplicial commutative $\S$-algebra. Then we have
a natural base change map $\_ \otimes^\bL_A B^\bullet \colon \D(A) \to \D(B^\bullet)$.
We denote its right adjoint by $\Tot_A$. 

Let $A  \to B$ be a map of commutative $\S$-algebras. We associate to this map the following
coaugmented cosimplicial algebra: for any $[n] \in \bigtriangleup$ we let $B^n$ be
the $[n]$-fold coproduct of $A \to  B$ in the category of commutative $\S$-algebras
under $B$. We refer to $A \to B^\bullet$ to the coaugmented cosimplicial algebra associated
to $A \to B$.

We will usually write $\unit$ for the tensor unit in $\Cpx(\Ab)^\integers$.
We will write $\integers$ for the complex in $\Cpx(\Ab)$ with $\integers$ in degree $0$.

Let $\caC$ be a category. For an object $M \in \caC^\integers$ and $r \in \integers$
we write $M(r)$ for the corresponding object of $\caC$ sitting in degree $r$.

We denote the two possible functors $\Cpx(\Cpx(\Ab)) \to \Cpx(\Ab)$
assigning to a double complex the total complex where sums resp. products are used
by $\Total^\oplus$ resp. $\Total^\Pi$.

For a simplicial or cosimplicial object $X$ of an additive category we denote
by $\UC X$ the unnormalized chain complex associated to $X$.

In the text we use that if we have a simplicial object $M_\bullet$ in $\Cpx(\Ab)$
then its homotopy colimit can be computed by $\Total^\oplus(\UC M_\bullet)$.
Similarly, for a cosimplicial object $M^\bullet$ in $\Cpx(\Ab)$ the homotopy limit
can be computed by $\Total^\Pi(\UC M^\bullet)$. Thus in particular
$\Tot_A$ is given by such a formula.

\section{Diagrams of model categories and cofinality}
\label{diag-mod-cat}

Let $I$ be a small category and
$\omega \colon I \to \ModCat$ be a pseudo-functor.
Here $\ModCat$ denotes the $2$-category of model categories where
the morphisms are the left Quillen functors and the $2$-morphisms
the natural isomorphisms.

\begin{definition}
The category of sections of the
fibered category over $I^\op$ corresponding to $\omega$
is denoted by $\Sect(\omega)$. If it exists,
we endow $\Sect(\omega)$ with the projective model
structure, i.e. the model structure
such that weak equivalences and fibrations
are objectwise.
\end{definition}

In what follows we suppose that the projective model
structures on the section categories exist,
which is for example the case if $\omega$
takes values in cofibrantly generated model categories.

An object $x \in \Ho \Sect(\omega)$ is
called {\em cartesian} if for every
map $f \colon i \to j$ in $I$ the natural
morphism $\bL \omega(f)(x_i) \to x_j$
in $\Ho \omega(j)$ is an isomorphism.
The full subcategory of $\Ho \Sect(\omega)$
of cartesian objects is denoted by $\Ho \Sect(\omega)_\cart$.

Let $F \colon I' \to I$ be a functor and
set $\omega' := \omega \circ F$. Then we have
a natural pullback functor $F^* \colon
\Sect(\omega) \to \Sect(\omega')$ which is
clearly a right Quillen functor.
Let $F_*$ be its left adjoint.

For any $i \in I$ we denote by $F/i$
the over category relative to the functor $F$.
We let $J_i \colon F/i \to I'$ be the forgetful functor
and $\omega_i$ the functor $\omega' \circ J_i$.
We have the restriction functor
$J_i^* \colon \Sect(\omega') \to \Sect(\omega_i)$,
which is a right Quillen functor.

\begin{lemma} \label{res-left-Quillen}
The restriction functor $J_i^*$ is also a
left Quillen functor.
\end{lemma}
\begin{proof}
By abstract nonsense $J_i^*$ has a right adjoint.
We show that the right adjoint of $J_i^*$ preserves fibrations
and trivial fibrations. By adjunction
it is enough to show that $J_i^*$
preserves (trivial) cofibrations of the form
$\Hom_{I'}(i',\_) \times f$ for $f$ a (trivial)
cofibration in $\omega'(i')$, where
$\Hom_{I'}(i',\_) \times \_$ denotes the left adjoint
to the evaluation functor $\ev_{i'} \colon
\Sect(\omega') \to \omega'(i')$.
But clearly $$J_i^*(\Hom_{I'}(i',\_) \times f)
\cong \coprod_{\alpha \in \Hom_I(F(i'),i)}
\Hom_{F/i}((i',\alpha), \_) \times f \text{,}$$
which shows the claim.
\end{proof}
\begin{corollary} \label{value-left-adjoint}
For an objectwise cofibrant section $x \in \Sect(\omega')$
the value
of $\bL F_*(x)$ at $i \in I$ is given as the
homotopy colimit of the diagram
$$\begin{array}{rccl}
D_i \colon & F/i & \to & \omega(i) \\
& (i',F(i') \overset{f}{\to} i) &
\mapsto & \omega(f)(x_{i'})
\end{array} \text{.}$$
\end{corollary}
\begin{proof}
The value of $F_*(x)$ is given as the strict
colimit of the above diagram.
If now $x$ is cofibrant then
by lemma (\ref{res-left-Quillen}) the above diagram
is also a cofibrant object in the diagram
category $\omega(i)^{F/i}$ with the projective
model structure, hence the colimit computes
the homotopy colimit.
\end{proof}

\begin{remark}
This is a model category theoretic proof (under the mild assumptions that the model
structures exist) of the section category version
of the general infinity category
statement that the left Kan extension
of a functor $I' \to K$ along a functor $I' \to I$
is computed by such (homotopy) colimits.
\end{remark}

Recall that a category is called {\em contractible}
if the realization of the nerve of the category is contractible.

\begin{lemma} \label{contr-htp-colimit}
Let $I$ be a contractible category
and $D \colon I \to \caC$ be a diagram in a
model category $\caC$ such that all transition maps
are weak equivalences. Then for any $i 
\in I$ the map from $D(i)$ to the homotopy colimit
of $D$ is also a weak equivalence.
\end{lemma}

\begin{proof}
Note that if we would know that the diagram were equivalent to a constant
diagram then the lemma would directly follow from \cite[Theorem 19.6.7 (1)]{hirschhorn}.

By looking at mapping spaces we reduce the problem to the dual statement
for simplicial sets, i.e. we have a diagram $D \colon I \to \sSet$
where the transition maps are weak equivalences and we want to prove that
the map from the homotopy limit to every $D(i)$ is a weak equivalence.
We apply the Bousfield-Kan spectral sequence \cite[Ch. XI, 7.1]{bousfield-kan}:
the $E_2$-term is given by the $\bR^p \lim \pi_q D$, $p \le q$. But for every $q$
the diagram $\pi_q D$ is constant over the contractible category $I$,
thus it follows from \cite[Theorem 19.6.7 (2)]{hirschhorn} that
$\bR^p \lim \pi_q D=0$ for $p>0$ (for $q=1$ one gives the argument using
\cite[Ch. XI, 7.2 (ii)]{bousfield-kan}). Thus the spectral sequence degenerates
and the result follows.
\end{proof}

Let $I$, $F$,
$\omega$ and $\omega'$ be as in the beginning of this section.
Recall from \cite[Definition 19.6.1]{hirschhorn} that the functor $F$ is called
{\em homotopy left cofinal} if for every object $i \in I$
the over category $F/i$ is contractible.

\begin{lemma} \label{cart-equiv}
Suppose that $F$ is homotopy left cofinal.
Then the functor $\bL F_*$ sends cartesian
objects to cartesian objects and the
adjoint functors $\bL F_*$ and $\bR F^*$
restrict to mutually inverse equivalences between
$\Ho \Sect(\omega')_\cart$ and $\Ho \Sect(\omega)_\cart$.
\end{lemma}
\begin{proof}
Let $x \in \Sect(\omega')$ and $i \in I$.
Suppose all the $x_{i'}$, $i' \in I'$, are cofibrant.
By corollary (\ref{value-left-adjoint}) the value of $\bL F_*(x)$ at $i$ is computed
as the homotopy colimit over the diagram
$D_i \colon F/i \to \omega(i)$, $(i',f \colon F(i') \to i)
\mapsto \omega(f)(x_{i'})$. Now if $x$ is cartesian
then all transition maps of the diagram $D_i$
are weak equivalences, and the index category $F/i$
is contractible by assumption, hence
by lemma (\ref{contr-htp-colimit}) the
maps $D_i((i',f)) \to \hocolim D_i$ are weak equivalences,
too. This shows the first statement.

The unit $\id \to \bR F^* \circ \bL F_*$ and the
counit $\bL F_* \circ \bR F^* \to \id$
are isomorphisms on cartesian objects by the
same argument as for the first part.
This completes the proof.
\end{proof}

\begin{lemma} \label{htp-type-over-cat}
Let $D$ be a small category and $F \colon \bigtriangleup \to D$
a functor. Then for any $d \in D$ the over
category $F/d$ has the homotopy type of the
simplicial set $[n] \mapsto \Hom_D(F([n]),d)$.
\end{lemma}
\begin{proof}
This is as in the proof of
\cite[Proposition 6.11]{dwyer-kan-function}.
\end{proof}

We will apply lemma (\ref{cart-equiv})
to the functors $\mathrm{diag} \colon \bigtriangleup
\to \bigtriangleup \times \bigtriangleup$
and $i \colon \bigtriangleup \to \bigtriangleup_*$.
We introduce the latter category.

First we denote by $\bigtriangleup_+$
the category of finite ordered sets $[p]$
for $p \in \{-1\} \cup \naturals$, so
that $[-1]$ is the empty ordered set.
We have a natural full embedding
$\bigtriangleup \hookrightarrow
\bigtriangleup_+$.

Next for any $[p] \in \bigtriangleup_+$
we denote by $[p]_*$ the ordered set
$[p] \sqcup \{*\}$, where we declare
the ordering by $p < *$. We view
$[p]_*$ as a {\em pointed} finite ordered set pointed
by $*$.
We let $\bigtriangleup_*$ be the category
of the pointed finite ordered sets
$[p]_*$, $[p] \in \bigtriangleup_+$,
with order preserving pointed maps.

So we have an embedding of categories
$\bigtriangleup \subset \bigtriangleup_*$
which is not full.

\begin{lemma} \label{pointed-contr}
For any $[q]_* \in \bigtriangleup_*$
the simplicial set $[p] \mapsto \Hom_{\bigtriangleup_*}
([p]_*, [q]_*)$
is contractible.
\end{lemma}
\begin{proof}
In fact any augmented simplicial set $K \colon
\bigtriangleup_+^\op \to \Set$ which can be
extended to a functor $\bigtriangleup_* \to \Set$
has the discrete homotopy type of $A:=K([-1])$.
Such an extension $\widetilde{K}$
gives the following data:
We denote by $\mathrm{cs}A$ the constant simplicial
set on $A$. The augmentation of $K$
is a map $\epsilon \colon K \to \mathrm{cs}A$. The
projections $[p]_* \to [-1]_*$ define
a section $s$ of $\epsilon$.
Furthermore we get a
simplicial homotopy $s \circ \epsilon \to \id$
as follows: A map $K \times \Delta[1] \to K$
is a family of maps $h_p^\alpha \colon K_p \to K_p$,
$\alpha \in \Delta[1]_p = \Hom_\bigtriangleup([p],[1])$.
Now for any such $\alpha$ we denote by
$r_\alpha$ the map $[p]_* \to [p]_*$ which
is the identity on $\alpha^{-1}(0)$ and
sends $\alpha^{-1}(1)$ to $*$.
Define $h_p^\alpha$ to be the map
$\widetilde{K}(r_\alpha)$.
Then it is easily checked that the
$h_p^\alpha$ fit together to a homotopy
from $s \circ \epsilon$ to $\id$.
\end{proof}

\begin{corollary} \label{incl-cofinal}
The inclusion functors
$\mathrm{diag} \colon \bigtriangleup \to
\bigtriangleup \times \bigtriangleup$ and
$i \colon \bigtriangleup \to \bigtriangleup_*$
are homotopy left cofinal.
\end{corollary}
\begin{proof}
Let $([n],[m]) \in \bigtriangleup \times \bigtriangleup$.
By lemma (\ref{htp-type-over-cat}) the over category
$\mathrm{diag}/([n],[m])$ has the homotopy type
of $\Delta[n] \times \Delta[m]$, which is contractible.
By the same lemma the over category
$i/[q]_*$ has the homotopy type of the simplicial
set $$[p] \mapsto \Hom_{\bigtriangleup_*}([p]_*,[q]_*) \text{,}$$
which is contractible by lemma (\ref{pointed-contr}).
\end{proof}

\begin{proposition} \label{cofinal-equiv}
Let $F \colon \bigtriangleup_* \to \ModCat$ be a diagram of model
categories (i.e. a pseudo-functor), let $F' \colon \bigtriangleup \to \ModCat$
be the restricted diagram. Then the canonical functor
$\Ho F([-1]_*) \to \Ho \Sect(F')_\cart$ is an equivalence.
\end{proposition}

\begin{proof}
Consider the following $2$-commutative diagram:
$$\xymatrix{\Ho \Sect(F')_\cart \ar[r] & \Ho \Sect(F)_\cart \\
\Ho F([-1]_*) \ar[u] \ar[ru]}$$
The inclusion $\{[-1]_*\} \to \bigtriangleup_*$ is homotopy left cofinal,
thus by lemma (\ref{cart-equiv}) the diagonal arrow is an equivalence.
By corollary (\ref{incl-cofinal}) the inclusion $\bigtriangleup \to \bigtriangleup_*$
is homotopy left cofinal thus by lemma (\ref{cart-equiv}) the horizontal arrow is an equivalence.
It follows that the vertical map is also an equivalence which was to be shown.
\end{proof}

\begin{corollary} \label{corollary:diag-equiv}
Let $F \colon \bigtriangleup \times \bigtriangleup \to \ModCat$ be a diagram of model
categories, let $F' \colon \bigtriangleup \to \ModCat$ be the composition of $F$ with
the diagonal $\bigtriangleup \to \bigtriangleup \times \bigtriangleup$. Then the canonical
functor $\Ho \Sect(F)_\cart \to \Ho \Sect(F')_\cart$ is an equivalence.
\end{corollary}

\begin{proof}
This follows from lemma (\ref{cart-equiv}) and corollary
(\ref{incl-cofinal}).
\end{proof}

\section{Affine derived group schemes}
\label{hgddf}

The definitions in this section go back to \cite{toen.habil}.

Let $R$ be a commutative $\S$-algebra.

We let $\alpha_i \colon [1] \to [n]$, $i=0,\ldots,n-1$, be the
map $0 \mapsto i$, $1 \mapsto i+1$. We let $s \colon [0] \to [1]$ be
the map $0 \mapsto 0$, $t \colon [0] \to [1]$, $0 \mapsto 1$.
Finally, let $c \colon [1] \to [2]$ be the map $0 \mapsto 0$,
$1 \mapsto 2$.

\begin{definition} \label{defi-derived-affine-gr-scheme}
An {\em affine derived groupoid} over $R$ is a cosimplicial
commutative $R$-$\S$-algebra $A^\bullet$ such that
\begin{enumerate}
\item the canonical map
$$A^1 \otimes^\bL_{t,A^0,s} A^1 \otimes^\bL \cdots \otimes^\bL_{t,A^0,s} A^1
\to A^n$$
($n$ tensor factors on the left hand side)
induced by $\alpha_0, \ldots, \alpha_{n-1}$ is an equivalence,
\item the map
$$A^1 \otimes^\bL_{s,A^0,s} A^1 \to A^2$$
induced by $c, \alpha_0$
is an equivalence.
\end{enumerate}

An {\em affine derived group scheme} over $R$ is a derived affine
groupoid over $R$ such that the map $R \to A^0$ is an equivalence.
\end{definition}

Note that the left hand sides in (1) and (2) for an affine derived group can be
written as an absolute tensor product of $A^1$'s.

\begin{proposition} \label{cech-der}
Let $A \to B$ be a cofibration of commutative $\S$-algebras. Let $A  \to B^\bullet$
be the corresponding coaugmented cosimplicial commutative $\S$-algebra. Then $B^\bullet$
is an affine derived groupoid over $A$ or equivalently over the initial algebra.
\end{proposition}

\begin{proof}
The first property follows from the equivalence
$$(B \otimes_A B)\otimes^\bL_{t,B,s} (B \otimes_A B)\otimes^\bL_{t,B,s} \cdots \otimes^\bL_{t,B,s} (B \otimes_A B)
\overset{\simeq}{\longrightarrow} B \otimes_A B \otimes_A \cdots \otimes_A B =
B^n,$$
where in the first entry $n$ tensor factors are used, and in the second entry $n+1$.
The second property follows from the fact that the map
$$(B \otimes_A B) \otimes^\bL_{s,B,s} (B \otimes_A B) \to B \otimes_A B \otimes_A B = B^2$$
induced by the maps $c, \alpha_0$ is an equivalence.
\end{proof}

\begin{proposition} \label{cech-der-equiv}
Let $R \to A$ be a map of commutative $\S$-algebras and $A \to B$ a cofibration of commutative $\S$-algebras.
Let $A  \to B^\bullet$ be the corresponding coaugmented cosimplicial commutative $\S$-algebra.
Suppose the map $R \to B$ is an equivalence. Then $B^\bullet$
is an affine derived group over $R$.
\end{proposition}

\begin{proof}
By proposition (\ref{cech-der}) $B^\bullet$ is an affine derived groupoid.
Since $R \to B=B^0$ is an equivalence the result follows.
\end{proof}

\begin{remark}
Let $R \to A$ be a cofibration of commutative $\S$-algebras. Factor the map
$A':= A \otimes_R A \to A$ into a cofibration $A' \to B$ followed by a weak equivalence
$B \to A$. Let $A' \to B^\bullet$ be the coaugmented cosimplicial algebra associated
to $A' \to B$. Then $B^\bullet$ is an affine derived group scheme over $A$, where
we consider $B^\bullet$ as lying over $A$ via any of the two canonical morphisms
$A \to A'$. This can be considered as giving the loop stack $\Spec(B^1) = \Spec(S^1 \otimes^\bL_R A)$
the structure of a derived group scheme over $\Spec(A)$. If now $A \to C$ is a cofibration
such that $R \to C$ is an equivalence then the pushforward of the loop stack
with respect to the map $A \to C$ gives the construction of the affine derived group scheme
given by proposition (\ref{cech-der-equiv}) applied to the datum $R \to A \to C$.
\end{remark}

Let $A^\bullet$ be an affine derived group scheme over some commutative $\S$-algebra
$R$. By a {\em representation} of $A^\bullet$ we understand a homotopy cartesian
$A^\bullet$-module $M^\bullet$. The tensor triangulated category of representations
of $A^\bullet$ is $\D(A^\bullet)_\cart$. A representation is called {\em perfect} if
it is an object of $\Perf(A^\bullet)$.

\begin{definition} \label{jhgyr}
A pro affine derived group scheme over $R$ is a functor $B_\_^\bullet \colon i \mapsto B_i^\bullet$ from
a filtered category $I$ to the category of cosimplicial commutative $R$-$\S$-algebras
such that for each $i \in I$ the cosimplicial algebra $B_i^\bullet$
is an affine derived group scheme over $R$.
\end{definition}
We will write $\text{``$\lim_i$''} B_i^\bullet$ for a pro affine derived group scheme.

The category of representations of a pro affine derived group scheme
$\text{``$\lim_i$''} B_i^\bullet$ is defined to be the $2$-colimit of the representation categories
of the individual affine derived group schemes $B_i^\bullet$, i.e. it is
$\text{$2$-$\colim$}_i \D(B_i^\bullet)_\cart$. This is a tensor triangulated category.
The category $\Perf(\text{``$\lim_i$''} B_i^\bullet)$ is defined to be
$\text{$2$-$\colim$}_i \Perf(B_i^\bullet)$. It is again a tensor triangulated category.

\section{Tate algebras and descent}

Let $A$ be a commutative $\S$-algebra in $\Cpx(\Ab)^\integers$.
We say that $A$ is of {\em Tate-type} if
\begin{enumerate}
\item $A(k) \simeq 0$ for $k>0$,
\item the map $\integers \to A(0)$ is a quasi-isomorphism. 
\end{enumerate}

We further say that $A$ is of {\em bounded} Tate-type
if $A(k)$ is cohomologically bounded from below for each $k<0$.
We say that $A$ is of {\em strict} Tate-type if $A(k)=0$
for $k>0$ and the map $\integers \to A(0)$ is an isomorphism.

It is clear that for any algebra $A$ of Tate-type there is a canonical
algebra $A'$ of strict Tate-type together with a quasi-isomorphism
$A' \to A$.

Let $A$ be an algebra of strict Tate-type. Then there is a canonical
augmentation $A \to \unit$ being the identity in Adams degree $0$.

We associate to this situation an affine derived group scheme in
the following way: factor the map $A \to \unit$ into a cofibration
$A \to B$ followed by a weak equivalence $B \to \unit$. Let $A \to B^\bullet$
be the coaugmented cosimplicial algebra associated to $A  \to B$.
Then by proposition (\ref{cech-der-equiv}) $B^\bullet$ is a derived
affine group scheme over $\unit$.

Notice that for any algebra $A$ of Tate-type there is a canonical
derived pushforward
$U \colon \D(A) \to \D(\unit)\simeq \D(\Ab)^\integers$
by first restricting to the quasi isomorphic
algebra of strict Tate-type $\D(A) \to \D(A')$ and then taking
the derived pushforward along the augmentation.

Let $A$ be an algebra of  Tate-type.
We denote by $\D(A)_{\mathrm{Aba}} \subset \D(A)$
the full tensor triangulated subcategory
consisting of modules $M$
such that $M(r) \simeq 0$ for $r$ big enough.
We call these modules {\em Adams bounded from above}.

\begin{lemma} \label{tate-conservative}
Let $A$ be an algebra of Tate-type.
Then the push forward
$$U \colon \D(A)_{\mathrm{Aba}} \to
\D(\unit) \simeq \D(\Ab)^\integers$$
is conservative, i.e. sends non-zero objects
to non-zero objects.
\end{lemma}

\begin{proof}
For the argument we replace the $\S$-algebra
$A$ by an equivalent associative algebra,
also denoted by $A$, since then it is easier to
compute the derived push forward. This can
be done in the way that $A(r) =0$ for $r>0$
and $A(0) = \integers$.
The algebra $A$ has then the canonical augmentation
$e \colon A \to \integers$.
We denote now by $\Mod(A)$ the category of left
$A$-modules.

Now let $M \in \Mod(A)$ be a module which
is Adams bounded from
above and non-equivalent to zero.
Let $r_0$ be the biggest integer such that
$M(r_0)$ is non-equivalent to $0$. Assume
$M$ is cofibrant for the projective model structure
on $\Mod(A)$, so we can compute $U(M)$
by $\unit \otimes_A M=:M'$. Furthermore we can assume
$M(r) =0$ for $r > r_0$ (use the model structure
on $\{M \in \Mod(A) \; | \; M(r)=0, r > r_0\}$ and observe
that a cofibrant object for this model structure is also cofibrant
for the projective model structure on $\Mod(A)$).
Then clearly $M'(r) = 0$ for $r > r_0$ and
$M'(r_0) = M(r_0)$, hence
$M'$ is non-equivalent to $0$, too.
\end{proof}

\begin{lemma} \label{descent-section}
Let $f \colon A \to B$ be a cofibration of commutative $\S$-algebras
and suppose $f$ has a cosection $s$.
Let $A \to B^\bullet$ be the corresponding
coaugmented cosimplicial algebra.
Then the whole category of modules
over $A$ satisfies descent with respect to the
map $f$, i.e. the functor
$$\D(A) \to \D(B^\bullet)_\cart$$
is an equivalence.
\end{lemma}
\begin{proof}
Define a map $(I \setminus \{*\}) \otimes_A B
\to (J \setminus \{*\}) \otimes_A B$
for a map between pointed finite sets $\varphi \colon I \to J$
by sending the summands in $(I \setminus \{*\}) \otimes_A B$
corresponding to preimages of $*$ under $\varphi$ to
the target by first projecting down to $A$ via $s$
and then using the structure map.
On the remaining summands the map is defined
by application of $\varphi$.
By this procedure one gets a functor
$\Phi \colon \bigtriangleup_* \to \Alg(\caC)$,
$[p]_* \to [p] \otimes_A B$,
extending the coaugmented algebra $A \to B^\bullet$.

By proposition (\ref{cofinal-equiv}) it follows that the functor
$\D(A) \to \D(B^\bullet)_\cart$ is an equivalence,
which was to be shown.
\end{proof}

\begin{lemma} \label{descent-unit}
Let $f \colon A \to B$ be a cofibration of commutative $\S$-algebras
and $A \to B^\bullet$ be the corresponding
coaugmented cosimplicial algebra
(i.e. $B^n=B \otimes_A \cdots \otimes_A B$ ($n+1$ times)).
Let $B \to B'^\bullet$ be the pushforward
of $A \to B^\bullet$ along $f$.
Let $\C \subset \D(A)$ be a full triangulated subcategory such that
the following properties are fulfilled:
\begin{enumerate}
\item The functor $\D(B^\bullet) \to \D(B'^\bullet)$
preserves total objects for modules in $\D(B^\bullet)$
which are in the image of the composition
$\C \to \D(A) \to \D(B^\bullet)$, in the
sense that for
$M \in \C$, the natural map
$$\Tot_A(M \otimes_A^\bL B^\bullet) \otimes_A^\bL B
\to \Tot_B(M \otimes_A^\bL B^\bullet \otimes_A^\bL B)$$
is an isomorphism.
\item The functor $\Tot_A \circ (\_ \otimes^\bL_A B^\bullet)$
maps $\C$ to itself.
\item The composition $\C \to \D(A) \to \D(B)$
is conservative.
\end{enumerate}
Then $f$ satisfies descent for modules in $\C$ in the sense
that the unit for the adjunction between $\D(A)$
and $\D(B^\bullet)$ is an isomorphism on objects
from $\C$.
\end{lemma}

\begin{proof}
By properties (2) and (3) we can test the unit to be
an isomorphism on an object $M \in \C$
after applying $\bL f_*$.
By property (1) this amounts to saying
that $\bL f_* M$ satisfies descent
with respect to the coaugemented cosimplicial
algebra $B \to B'^\bullet$.
This is clearly fulfilled
by Lemma \ref{descent-section}.
\end{proof}

\begin{lemma} \label{descent-counit}
Let $f \colon A \to B$, $A \to B^\bullet$
and $B \to B'^\bullet$ be as in Lemma \ref{descent-unit}.
Let $\D \subset \D(B^\bullet)_\cart$ be a full triangulated
subcategory such that the following properties are fulfilled:
\begin{enumerate}
\item The functor $\D(B^\bullet) \to \D(B'^\bullet)$
preserves total objects for modules in $\D$
in the sense that for
$M \in \D$, the natural map
$$\Tot_A(M) \otimes_A^\bL B
\to \Tot_B(M \otimes_A^\bL B)$$
is an isomorphism.
\item The functor $(\_ \otimes^\bL_A B^\bullet) \circ \Tot_A$
maps $\D$ to itself.
\end{enumerate}
Then $f$ satisfies descent for modules in $\D$ in the sense
that the counit for the adjunction between $\D(A)$
and $\D(B^\bullet)$ is an isomorphism on objects
from $\D$.
\end{lemma}
\begin{proof}
Note first that the functor $\D(B^\bullet)_\cart \to \D(B'^\bullet)_\cart$
is conservative since the map $B \to B'=B \otimes_A B$ has a retraction.
Thus by property (3) we can test if the counit is
an isomorphism on an object $M \in \D$
after applying $\_ \otimes^\bL_A B = \_ \otimes^\bL_{B^\bullet}
B'^\bullet$.
By property (1) this amounts to saying
that $M\otimes^\bL_A B$ satisfies descent
with respect to the coaugmented cosimplicial
algebra $B \to B'^\bullet$,
which is again fulfilled
by Lemma \ref{descent-section}.
\end{proof}

\begin{lemma} \label{real-alg-commute}
Let $A_\bullet$ be a simplicial commutative $\S$-algebra in $\Cpx(\Ab)^\integers$
and $|A|_{\mathrm{comm}}$ be its realization
in commutative $\S$-algebras which is equivalent to
the homotopy colimit of $A_\bullet$.
Let $A^\sharp_\bullet$ be the underlying simplicial
module and $|A |_{\mathrm{comm}}^\sharp$
the underlying module of $|A|_{\mathrm{comm}}$.
Then the natural map
$$|A^\sharp_\bullet | \to |A |_{\mathrm{comm}}^\sharp$$
is an isomorphism in $\D(\Ab)^\integers$.
\end{lemma}
\begin{proof}
This follows from \cite[VII., Prop. 3.3]{ekmm}.
\end{proof}

\begin{lemma} \label{htp-b1}
Let $A$ be an algebra of strict Tate-type and
$e \colon A \to \unit$ the corresponding augmentation.
Let $A \to B \to \unit$ be a factorization of $e$
into a cofibration followed by a weak equivalence.
Let $B^\bullet$ be the cosimplicial algebra
associated to $A \to B$. Then the homotopy type
of the algebra $B^1$ is the same as the realization of the
simplicial algebra $(A \otimes^\bL \Delta^1)
\otimes^\bL_{A \otimes^\bL \partial \Delta^1}
(\unit \otimes^\bL \partial \Delta^1)$, where the map $A \otimes^\bL \partial \Delta^1
\to \unit \otimes^\bL \partial \Delta^1$ is $e \otimes^\bL e$.
\end{lemma}

\begin{proof}
We first note that the realization of the simplicial algebra
$A \otimes^\bL \Delta^1$ is $A$. The simplicial algebra
$A \otimes^\bL \Delta^1$ is a simplicial object in
$A \otimes^\bL \partial \Delta^1 \simeq A \otimes^\bL A$-algebras.
Realization of such algebras commutes with pushforward
along any map $A \otimes^\bL A \to C$ of algebras.
Thus we see that the realization of $(A \otimes^\bL \Delta^1)
\otimes^\bL_{A \otimes^\bL \partial \Delta^1}
(\unit \otimes^\bL \partial \Delta^1)$ is equivalent
to $A \otimes^\bL_{A \otimes^\bL A} (\unit \otimes^\bL \unit)$ which
in turn is equivalent to $\unit \otimes^\bL_A \unit$.
\end{proof}

\begin{corollary} \label{htp-module-b1}
Let $A$ be an algebra of strict Tate-type and
$e \colon A \to \unit$ the corresponding augmentation.
Let $A \to B \to \unit$ be a factorization of $e$
into a cofibration followed by a weak equivalence.
Let $B^\bullet$ be the cosimplicial algebra
associated to $A \to B$. Then the homotopy type
of the underlying module of
$B^1$ is the same as the realization of the simplicial module
$[n] \mapsto A^{\otimes^\bL \Delta^1(n)} \otimes^\bL_{A \otimes^\bL A}
(\unit \otimes^\bL \unit)$.
\end{corollary}
\begin{proof}
This follows from Lemma \ref{htp-b1} with Lemma \ref{real-alg-commute}.
\end{proof}

\begin{lemma} \label{htp-b1-descr}
Let $A$ be an algebra of strict Tate-type and
$e \colon A \to \unit$ the corresponding augmentation.
Let $A'$ be a replacement of $A$ as a strictly associative
algebra of strict Tate-type such that for
any $k$ the complex $A'(k)$ is cofibrant
(which always exists by using e.g. a model
structure on augmented strictly associative
algebras).
Set $\overline{A}:= A'/\unit$.
Then the realization of the simplicial module
$[n] \mapsto A^{\otimes^\bL \Delta^1(n)} \otimes^\bL_{A \otimes^\bL A}
(\unit \otimes^\bL \unit)$ is computed as the
total complex of a bicomplex of the form
$$\unit \leftarrow \overline{A} \leftarrow
\overline{A}^{\otimes 2} \leftarrow
\overline{A}^{\otimes 3} \leftarrow \cdots \text{,}$$
where in the last line the bare tensor product
in $\Cpx(\Ab)^\integers$ is used.
\end{lemma}

\begin{proof}
In fact applying degreewise the normalized
associated chain complex to
the two-sided simplicial Bar-complex
$B_*(\unit,A',\unit)$ which is defined
analogously to \cite[IV., Definition 7.2]{ekmm}
yields exactly a bicomplex of the type indicated.
Comparing the two-sided Bar-complex
for associative $\S$-algebras and
strictly associative algebras yields an equivalence
between the realization of $B_*(\unit,A',\unit)$
and that of the simplicial module
$[n] \mapsto A^{\otimes^\bL \Delta^1(n)} \otimes^\bL_{A \otimes^\bL A}
(\unit \otimes^\bL \unit)$.
\end{proof}

\begin{corollary} \label{uniform-bound}
Let $A$ be an algebra of bounded strict Tate-type and
$e \colon A \to \unit$ the corresponding augmentation.
Let $A \to B \to \unit$ be a factorization of $e$
into a cofibration followed by a weak equivalence.
Let $B^\bullet$ be the cosimplicial algebra
associated to $A \to B$. Then all algebras $B^n$
are of bounded Tate-type. Moreover for any $k$
the complexes $B^n(k)$ are uniformly in $n$
cohomologically bounded from below.

The same is true for the cobase change $B'^\bullet$
of $B^\bullet$ along $A \to B$.
\end{corollary}

\begin{proof}
By Corollary \ref{htp-module-b1} and
Lemma \ref{htp-b1-descr} the underlying homotopy
type of $B^1$ is given as the total complex $C$
of a bicomplex of the form
$$\unit \leftarrow \overline{A} \leftarrow
\overline{A}^{\otimes 2} \leftarrow
\overline{A}^{\otimes 3} \leftarrow \cdots \text{,}$$
where $\overline{A}$ sits in strictly negative Adams
degree and is in each Adams degree cohomologically
bounded from below.
It follows that $C$ has $\integers$ in Adams degree $0$
and is $0$ in positive Adams degrees. Moreover
the contribution to $C$ in Adams degree $k<0$
comes from at most $-k$ tensor factors of some
$\overline{A}(l)$ with $l \ge k$.
Using that $\mathrm{Tor}^\integers_i(M,N)=0$ for $i>1$, $M,N$ abelian groups,
it follows that $C$, and hence also $B^1$, is in each
Adams degree cohomologically bounded from below.
This shows that $B^1$ is of Tate-type.
By the Segal condition $B^n$ is equivalent
to $(B^1)^{\otimes^\bL n}$. Again it follows that
$B^n$ is contractible in positive Adams degrees,
equivalent to $\integers$ in Adams degree $0$
and in each (negative) Adams degree cohomologically
bounded from below. Hence $B^n$ is of Tate-type.
For $k>0$ the contribution to $B^n(k)$
in $(B^1)^{\otimes^\bL n}$ only comes from
at most $-k$ tensor factors  of $B^1$
in Adams degrees $l$ with
$0 > l \ge k$, hence the $B^n(k)$ are uniformly in $n$
cohomologically bounded from below.

The statement for the $B'^\bullet$ follows from
the fact that $B'^n \simeq B^{n+1}$.
\end{proof}

\begin{lemma} \label{hgfsgh}
Let $M^\bullet$ be a cosimplicial object in $\Cpx(\Ab)$ and suppose there
is an $n_0 \in \integers$ such that for each $k \ge 0$ we have
$H^n(M^k)=0$ for $n < n_0$. Let $k_0 \ge 0$.
Let $C$ be the total complex corresponding to
the double complex which is associated
(via the unnormalized chain complex construction)
to the truncated cosimplicial object
$(M^k)_{k \le k_0}$. Then the natural map $\Tot(M^\bullet) \to C$
induces isomorphisms $H^n(\Tot(M^\bullet)) \to H^n(C)$ for
$n \le k_0 + n_0$.
\end{lemma}
\begin{proof}
By applying the good truncation functor there is a cosimplicial complex
$N^\bullet$ such that the complexes $N^k$ are uniformly in $k$
bounded from below (for the cohomological indexing) and a quasi isomorphism
$M^\bullet \to N^\bullet$. Thus the map $\Tot(M^\bullet) \to \Tot(N^\bullet)$
is a quasi isomorphism. Consider the filtration on $\Tot(N^\bullet)$ induced
by the subcomplexes coming from a cosimplicial degree bounded from below
by a fixed degree. Then the associated spectral sequence strongly converges
to the cohomology of $\Tot(N^\bullet)$ by the boundedness condition.

The strongly convergent spectral sequence for $C$ coincides
with this spectral sequence for fixed total degree $\le k_0 + n_0$
by the assumption, thus the result.
\end{proof}

We note that the notion of a (bounded) algebra of Tate type
also makes sense in the category of Adams graded graded abelian groups.
For example the cohomology of a (bounded) algebra of Tate type will be such a (bounded)
algebra of Tate type.

\begin{lemma} \label{juhtgd}
a) Let $A$ be a commutative algebra in Adams graded graded abelian groups which is
of Tate type. Let $M$ be an $A$-module which is Adams bounded from above
(i.e. there is an $N$ such that $M(k)=0$ for $k > N$). Let $s \in \naturals_{> 0}$.
Suppose there is $B \in \naturals$ such that $A(k)^i=0$ for $k \ge -s$ and
$i < -B$ and suppose there is $B' \in \integers$ such that $M(k)^i=0$
for $N \ge k \ge N-s$ and $i< B'$. Then there is an $A$-free resolution
$$\cdots \to P_i \to P_{i-1} \to \cdots \to P_1 \to P_0 \to M$$
with the following properties:
\begin{enumerate}
\item if $i \ge 2s$ we have $P_i(N)= \cdots =P_i(N-s)=0$,
\item for $2s > i \ge 0$ we have $P_i(k)^l=0$ if $N \ge k \ge N-s$ and $l < B' -2sB$,
\item for $2s > i \ge 0$ the generators of $P_i$ which are in Adams degrees $N-s$ up to $N$
lie all in degree $\ge B'- (2s-1)B$.
\end{enumerate}
b) Let the notation be as in a). Let $M_2$ be Adams bounded from above by $N_2$ and suppose
there is $B'' \in \integers$ such that $M_2(k)^i=0$
for $N \ge k \ge N-s$ and $i< B''$. Then $(M_2 \otimes_A P_i)(k)^l=0$ 
if $N+N_2 \ge k \ge N+N_2-s$ and $l < B'+B'' -(2s-1)B$.

\end{lemma}
\begin{proof}
a) For an $A$-module $M'$ we let $F(M')$ be the free $A$-module on the
elements of $A$, i.e. for each $i,j \in \integers$
and $x \in M'(i)^j$ there is a copy $\Sigma^{i,j}A$ in $F(M')$.
Inductively we construct $P_i$ as follows: We set $P_0:=F(M)$.
Note that the graded groups $P_0(N)^i, \ldots, P_0(N-s)^i$ are
zero for $i< B' -B$.
 
Let $K_0 =\ker(F(M) \to M)$. Then $K_0(N)$ is a graded abelian group consisting
of free abelian groups (recall $A$ is an Adams algebra and in particular $A(0)$
is the graded abelian group $\integers$ sitting in degree $0$).
We let $Q=\bigoplus_{i \in \integers} \Sigma^{i,N} A \otimes K_0(N)^i$.
There is a canonical map $Q \to K_0$ inducing an isomorphism in Adams degree $N$.
Let $K_0'$ be the same module as $K_0$ except that the $N$-th
Adams degree is set to $0$. Let $Q'= F(K_0')$. Then there is a canonical map
$Q' \to K_0$. The induced map $Q \oplus Q' \to K_0$ is a surjection and we let
$P=Q \oplus Q'$. Thus $P_1 \to P_0 \to M$ is exact.

We construct $P_2$ as $F(\ker(P_1 \to P_0))$. We construct $P_3$ as we constructed
$P_1$ using the fact that $\ker(P_2 \to P_1)(N-1)$ is a graded free abelian group.
In this way we get a resolution $P_\bullet \to M$ such that the $P_i$ are
constructed in a $2$-periodic way.
The vanishing and boundedness results claimed are easy consequences of the construction
of this resolution.

b) follows from item a)(3).
\end{proof}

\begin{lemma} \label{main-tot}
Let $A \to C$ be a map of bounded Tate algebras in $\Cpx(\Ab)^\integers$
and let $M^\bullet$ be a cosimplicial $A$-module.
Let $L^\bullet := M^\bullet \otimes^\bL_A C$.
Suppose that there exists $N \in \integers$ such that
$M^i(k) \simeq 0$ for all $i \ge 0$ and $k > N$
and
that for any $k \le N$ there exists $n_0 \in \integers$
such that $H^n(M^l(k))=0$ for all $n < n_0$ and all $l \ge 0$.
Then $L^\bullet$ satisfies the same boundedness conditions as $M^\bullet$
(with possibly other constants) and
the natural map $\Tot_A(M^\bullet) \otimes_A^\bL C
\to \Tot_C(L^\bullet)$ is
an equivalence.
\end{lemma}

\begin{proof}
We first note that it makes sense to talk about the total object of a truncated
cosimplicial $A$-module, e.g. we can consider $\Tot(M^{\le k})$. The same applies
to $L^\bullet$. Clearly we have that the map
$$\Tot_A(M^{\le k_0}) \otimes_A^\bL C \to \Tot_C(L^{\le k_0})$$
is an equivalence (because there are only finitely many contributions for each total degree).
We let $F$ denote the homotopy fiber of the map
$\Tot_A(M^\bullet) \to \Tot_A(M^{\le k_0})$.
Fix $s \in \naturals_{> 0}$. By assumption there exists
$n_0 \in \integers$ such that $H^n(M^l(k))=0$ for $N \ge k \ge N-s$, $l \ge 0$ and $n<n_0$.
By lemma (\ref{hgfsgh}) it follows that the map $\Tot_A(M^\bullet) \to \Tot_A(M^{\le k_0})$
induces an isomorphism in Adams degrees $N-s$ up to $N$ and cohomological degree
$\le k_0 + n_0$. Thus in Adams degree $N-s$ up to $N$ the cohomology of the fiber $F$ vanishes
in degrees $\le k_0 + n_0$. Set $B'':=k_0+n_0+1$.

We now invoke the strongly convergent K\"unneth spectral sequence for
$F \otimes_A^\bL C$:
$$E^2_{p,q,r}= \mathrm{Tor}^{A_{**}}_p(F_{**},C_{**})_{(q,r)}
\Longrightarrow (F\otimes_A^\bL C)_{(p+q,r)}.$$
Here the $**$-notation denotes homology.
To compute the Tor-term we use the free resolution $P_\bullet \to C_{**}$ provided
by lemma (\ref{juhtgd}). It is then given by the homology of the complex
$F_{**} \otimes_{A_{**}} P_\bullet$. The group $(F_{**} \otimes_{A_{**}} P_i)(k)^j$ contributes
to $$H^{j-i}((F\otimes_A^\bL C)(k)).$$ We will be interested in these contributions
only in Adams degrees from $N-s$ up to $N$. Let
$B \in \naturals$ such that $H^i(A(k))=0$ for $0 \ge k \ge -s$ and $i < -B$ and
let $B' \in \integers$ such that $H^i(C(k))=0$ for $0 \ge k \ge -s$ and $i < B'$.
Then lemma (\ref{juhtgd})(b) tells us that $(F_{**} \otimes_{A_{**}} P_i)(k)^j=0$,
$N \ge k \ge N-s$, if either $i \ge 2s$ or $j < B' + B'' -(2s-1)B$. These
considerations give us that $H^{i}((F\otimes_A^\bL C)(k))=0$ for
$N \ge k \ge N-s$ and $i < B' + B'' -(2s-1)B -2s$.

Thus the map $$\Tot_A(M^\bullet) \otimes_A^\bL C \to \Tot_A(M^{\le k_0}) \otimes_A^\bL C$$
induces an isomorphism in Adams degree $N-s$ up to $N$ in cohomological degree
\begin{equation}
\label{hngfd}
< B' + k_0 + n_0 -(2s-1)B -2s.
\end{equation}
For fixed $s$ we can vary $k_0$ to increase
this upper bound.

We now check the boundedness conditions for $L^\bullet$. Clearly it is also
Adams bounded from above by $N$. The cohomological bounds for $L^\bullet$
follow by an analogous spectral sequence argument as above.

We now consider the commutative square
$$\xymatrix{\Tot_A(M^\bullet) \otimes_A^\bL C \ar[r] \ar[d] & \Tot_C(L^\bullet) \ar[d] \\
\Tot_A(M^{\le k_0}) \otimes_A^\bL C \ar[r] & \Tot_C(L^{\le k_0}).}$$

By what we have shown it follows that for fixed $s$ all maps induce
isomorphisms on cohomology in Adams degree $N-s$ up to $N$ provided
the cohomological degree is small. Increasing $k_0$ the explicit upper bound
(\ref{hngfd}) shows that the upper
horizontal map is a quasi isomorphism in these Adams degrees. But $s$ was picked arbitrarily
(the bounds $n_0$, $B$ and $B'$ depend upon $s$), thus the result.  
\end{proof}

\begin{lemma} \label{coh-connected}
Let $M^\bullet$ be a cosimplicial object in $\Cpx(\Ab)$. Suppose there is an $n_0 \in \integers$
such that $H^n(M^i)=0$ for every $n \le n_0$ and $i \ge 0$. Then $\Tot(M^\bullet)$ is
cohomologically bounded from below.
\end{lemma}
\begin{proof}
Apply the good truncation to $M^\bullet$ at the place $n_0$: set $N^{i,n}=0$
for $n < n_0$ $N^{i,n_0}=M^{i,n_0}/\ker(d)$ ($d$ the differential $M^{i,n_0} \to M^{i,n_0+1}$)
and $N^{i,n}=M^{i,n}$ for $n>n_0$. Then there is a quasi isomorphism $M^\bullet \to N^\bullet$.
Clearly $\Tot(N^\bullet)$ is cohomologically bounded from below. But
$\Tot(M^\bullet) \to \Tot(N^\bullet)$ is a quasi isomorphism.
\end{proof}

In the following we suppose $A$ is of bounded strict Tate type. We use the notations
from the previous statements, i.e. we let
$e \colon A \to \unit$ be the corresponding augmentation
and $A \to B \to \unit$ a factorization of $e$
into a cofibration followed by a weak equivalence.

We now introduce subcategories of $\D(A)$ and $\D(B^\bullet)$
which will restrict to an equivalence.
Let $\C \subset \D(A)_{\mathrm{Aba}}$
be the full triangulated subcategory of modules which
are cohomologically bounded from below in each
Adams degree. Let $\D \subset \D(B^\bullet)_\cart$
be the full triangulated subcategory of modules $M^\bullet$,
such that $M^0$ is Adams bounded from above and
cohomologically bounded from below in each Adams degree.

\begin{lemma} \label{D-uniformly}
Let $M^\bullet \in \D$. Then in each Adams degree $M^n$ is uniformly in $n$ cohomologically bounded
from below.
\end{lemma}
\begin{proof}
Since $M^\bullet$ is cartesian we have $M^n \simeq M^0 \otimes^\bL B^n$.
By corollary (\ref{uniform-bound}) the property in question is fulfilled for
$B^\bullet$. Suppose $M$ is trivial in Adams degree $> k_0$.
Fix $k \ge 0$. Then the contribution to Adams degree $k_0 -k$ in $M^n$
is only from $M^0(k_0),\ldots,M^0(k_0-k)$ and $(B^n)(0),\ldots,(B^n)(-k)$.
The claim follows.
\end{proof}

\begin{lemma} \label{C-D}
The canonical functor $\D(A) \to \D(B^\bullet)$ sends $\C$ to $\D$.
\end{lemma}
\begin{proof}
We have to check that the functor $e \colon \D(A) \to \D(\unit)$ given by push forward
along the augmentation sends modules from $\C$ to modules which are Adams bounded from above
and in each Adams degree cohomologically bounded from below. Let $M \in \C$.
We can write $e(M)=A \otimes^\bL_{A \otimes^\bL A} (M \otimes \unit)$. This in turn is the realization
of a simplicial module $[n] \mapsto M \otimes^\bL A^{\otimes^\bL n}
\otimes^\bL \unit$. As in lemma (\ref{htp-b1-descr})
this realization can be computed as the total complex $C$ of a bicomplex
$$M \leftarrow M \otimes \overline{A} \leftarrow
M \otimes \overline{A}^{\otimes 2} \leftarrow
M \otimes \overline{A}^{\otimes 3} \leftarrow \cdots$$
Suppose $M$ is trivial in Adams degree $> k_0$.
Fix $k \ge 0$. The module $M \otimes \overline{A}^{\otimes l}$ only contributes
to Adams degree $k_0 - k$ in $C$ if $l \le k$ and then only the entries of $\overline{A}$
in Adams degree $\ge -k$ and the $M(k_0), \ldots , M(k_0-k)$.
This shows that $C$ is cohomologically bounded from below
in each Adams degree.
\end{proof}

\begin{lemma} \label{D-C}
The canonical functor $\D(B^\bullet) \to \D(A)$ sends $\D$ to $\C$.
\end{lemma}
\begin{proof}
Let $M^\bullet \in \D$. Suppose $M^0$ is trivial in Adams degree $> k_0$.
Then each $M^n$ is trivial in Adams degree $> k_0$, since $M^n \simeq M^0 \otimes^\bL B^n$
and $B^n$ is trivial in positive Adams degree. Thus $\Tot_A(M^\bullet)$ is trivial
in Adams degree $> k_0$ which shows that the image of $M^\bullet$ lies
in $\D(A)_{\mathrm{Aba}}$.

By lemma (\ref{D-uniformly}) $M^n$ is in each Adams degree cohomologically bounded from
below, uniformly in $n$. Lemma (\ref{coh-connected}) now shows that $\Tot_A(M^\bullet)$ is in each Adams
degree cohomologically bounded from below.
\end{proof}

\begin{lemma} \label{satisfies-unit}
The subcategory $\C$ satisfies the conditions of lemma (\ref{descent-unit}).
\end{lemma}
\begin{proof}
By lemma (\ref{tate-conservative}) property (3) of lemma (\ref{descent-unit}) is
fulfilled. Property (2) follows from lemmas (\ref{C-D}) and (\ref{D-C}).
For property (1) we use lemma (\ref{main-tot}). The assumptions
of lemma (\ref{main-tot}) are fulfilled by lemma (\ref{C-D}) and lemma (\ref{D-uniformly}).
This shows the claim.
\end{proof}

\begin{lemma} \label{satisfies-counit}
The subcategory $\D$ satisfies the conditions of lemma (\ref{descent-counit}).
\end{lemma}

\begin{proof}
Property (2) follows from lemmas (\ref{D-C}) and (\ref{C-D}).
For property (1) we want to again apply lemma (\ref{main-tot}). The assumptions
are fulfilled by lemma (\ref{D-uniformly}).
This shows the claim.
\end{proof}


\begin{theorem} \label{C-D-equiv} Suppose $A$ is of bounded Tate type.
Then the adjunction $$\D(A) \rightleftarrows \D(B)_\cart$$ restricts to an
equivalence $\C \overset{\sim}{\longrightarrow} \D$.
\end{theorem}

\begin{proof}
By lemmas (\ref{C-D}) and (\ref{D-C}) the adjunction restricts to an
adjunction between $\C$ and $\D$. By lemmas (\ref{satisfies-unit}) and
(\ref{satisfies-counit}) we can apply lemmas (\ref{descent-unit}) and
(\ref{descent-counit}) to conclude that the unit and counit of the
induced adjunction are isomorphisms. This implies the claim.
\end{proof}

We want to see that the equivalence between $\C$ and $\D$ further restricts
to an equivalence between $\Perf(A)$ and $\Perf(B^\bullet)$. Define
$\Perf'(B^\bullet)$ to be the full triangulated subcategory of $\D(B^\bullet)_\cart$
generated by the trivial one-dimensional representation $B^\bullet$. Then
clearly $\Perf'(B^\bullet) \subset \Perf(B^\bullet)$, and the equivalence
between $\C$ and $\D$ restricts to an equivalence $\Perf(A) \overset{\sim}{\to}
\Perf'(B^\bullet)$. Note that in order to get this equivalence we would only
have needed the first half of the above arguments about proving that
the unit of the adjunction in question is an isomorphism.

The next result does not require $B^\bullet$ to come from an algebra of bounded Tate type.
Note that the categories $\Perf(B^\bullet)$ and $\Perf'(B^\bullet)$ make sense for
arbitrary affine derived group schemes $B^\bullet$.

\begin{proposition} \label{true-perf}
Suppose $B^\bullet$ is an affine derived group scheme such that $B^1$ is an
algebra of Tate type. Then
we have $\Perf'(B^\bullet)=\Perf(B^\bullet)$.
\end{proposition}
\begin{proof}
Let $M^\bullet \in \Perf(B^\bullet)$, i.e. $M^\bullet \in \D(B^\bullet)_\cart$ and
$M^0$ is a perfect object in $\D(\Ab)^\integers$. This means that $M^0$ only lives
in finitely many Adams degrees and that the complex in each Adams degree is perfect.
We can assume without loss of generality that $M^\bullet$ is non-trivial.
Let $\varphi \colon \D(\unit) \rightleftarrows \D(B^\bullet)_\cart \colon \psi$ be the adjunction
where $\varphi$ assigns to a module $N$ the trivial $B^\bullet$-representation on $N$
and $\psi$ is the right adjoint of $\varphi$. 

Let $n_0$ be the biggest Adams degree in which $M^0$ is non-trivial and $m_0$ the smallest
such integers. We show that $M^\bullet$ lies in the image of $\Perf'(B^\bullet) \to \Perf(B^\bullet)$
by induction on $n_0 - m_0$.
Consider the counit $\varphi(\psi(M^\bullet)) \to M^\bullet$. We claim that the map
$\varphi(\psi(M^\bullet))^0 \to M^0$ is an isomorphism in Adams degree $n_0$.

Indeed, since $B^1$ is of Tate type, restricting the representation $M^\bullet$ to the trivial group, i.e. pushing
forward along $\D(B^\bullet)_\cart \to \D(\unit^\bullet)_\cart$, does not change the
Adams degree $n_0$ part of $M^\bullet$, and the counit for the adjunction
$\D(\unit) \rightleftarrows \D(\unit^\bullet)_\cart$ is an isomorphism.

We let $N(n_0)=\psi(M^\bullet)(n_0)$ and $N(n)=0$ for $n \neq 0$. We let
$f \colon N \to \psi(M^\bullet)$ be the canonical map in $\D(\Ab)^\integers$.
Then the composition $$g \colon \varphi(N) \overset{\varphi(f)}{\longrightarrow} \varphi(\psi(M^\bullet))
\to M^\bullet$$
is an isomorphism in Adams degree $n_0$. By induction hypothesis (or in the case $n_0=m_0$ since $g$
is an isomorphism) the cofiber of $g$ lies in the image of $\Perf'(B^\bullet) \to \Perf(B^\bullet)$,
hence we are done.
\end{proof}

\begin{theorem} \label{perf-descent-thm}
Let $A$ be of bounded Tate type. Then the canonical map $\Perf(A) \to \Perf(B^\bullet)$
is an equivalence.
\end{theorem}

\begin{proof}
This follows from theorem (\ref{C-D-equiv}) and proposition (\ref{true-perf}).
\end{proof}

\section{General coefficients}

In this section we want to generalize the previous results to general coefficients.
For that let $R$ be a commutative ring and let $A$ be an algebra of bounded Tate type.
We denote by $A_R$ a (derived) base change of $A$ to $R$, i.e. $A_R=A \otimes^\bL R$.
Let $A' \to A$ be the replacement which is of strict Tate type. Factor
the augmentation $A' \to \unit$ into a cofibration $A \to B$ followed by an equivalence
$B \to \unit$. Let $A \to B^\bullet$ be the coaugmented cosimplicial algebra associated
to $A \to B$. Let $B^\bullet_R :=B^\bullet \otimes^\bL R$.
Let $$b_R \colon \D(A) \rightleftarrows \D(A_R) \colon r_R$$ and
$$b_R' \colon \D(B^\bullet) \rightleftarrows \D(B_R^\bullet) \colon r_R'$$ be the obvious adjunctions.
Recall the subcategories $\C \subset \D(A)$ and $\D \subset \D(B^\bullet)_\cart$
defined in the previous section. We let $\C_R \subset \D(A_R)$ be the full subcategory
of objects $X$ such that $r_R(X) \in \C$ and $\D_R \subset \D(B_R^\bullet)_\cart$
the full subcategory of objects $X$ such that $r_R'(X) \in \D$.

\begin{theorem} \label{R-C-D} Suppose $A$ is of bounded Tate type.
The adjunction $\D(A_R) \rightleftarrows \D(B_R^\bullet)$ restricts to an
equivalence $\C_R \to \D_R$.
\end{theorem}

\begin{proof}
The diagrams
$$\xymatrix{\D(A_R) \ar[r] \ar[d]^{r_R} & \D(B_R^\bullet) \ar[d]^{r_R'} \\
\D(A) \ar[r] & \D(B^\bullet)}$$
and
$$\xymatrix{\D(A_R) \ar[d]^{r_R} & \D(B_R^\bullet) \ar[d]^{r_R'} \ar[l] \\
\D(A) & \D(B^\bullet) \ar[l]}$$
are $2$-commutative. Moreover the functors $r_R$ and $r_R'$ detect
isomorphisms. The claim follows then from theorem (\ref{C-D-equiv}).
\end{proof}

Let $\Perf(A_R) \subset \D(A_R)$ be the full subcategory of perfect $A_R$-modules and
let $\Perf(B_R^\bullet) \subset \D(B_R^\bullet)_\cart$ be the full subcategory of
$B_R^\bullet$-modules $M^\bullet$ such that each $M^n$ is a perfect $B_R^n$-module,
or equivalently such that $M^0$ is a perfect $B^0 \simeq R$-module.

Note that $\Perf(A_R) \subset \C_R$ and $\Perf(B_R^\bullet) \subset D_R$ (for the first inclusion we use
that $\mathrm{Tor}^\integers_i(M,N)=0$ for $i \ge 2$, $M$ and $N$ abelian groups).

\begin{theorem} \label{R-perf-equiv}
Suppose $A$ is of bounded Tate type. Then the natural functor
$\Perf(A_R) \to \Perf(B_R^\bullet)$ is an equivalence.
\end{theorem}

\begin{proof}
By theorem (\ref{R-C-D}) we are left to prove that the functor
in question is essentially surjective.
This works exactly as in the proof of proposition (\ref{true-perf}).
\end{proof}

We now turn to the case where $A$ is given over a commutative base ring $R$.
So suppose $A$ is a commutative $\S$-algebra in $\Cpx(R)^\integers$.
The definition of Tate type, strict Tate Type and bounded Tate type works
as in the case over the integers.

We suppose from now on that $A$ is of bounded Tate type. 
Without loss of generality we can assume that $A$ is also of strict Tate type.

Factor the canonical augmentation $A \to \unit_R$ into a cofibration
$A \to B$ followed by a weak equivalence $B \to \unit_R$.
Let $A \to B^\bullet$ be the coaugmented cosimplicial algebra associated
to $A \to B$. As in the case over the integers we can define
the full subcategories $\C \subset \D(A)$ and $\D \subset \D(B^\bullet)_\cart$.

\begin{theorem} \label{R-C-D-equiv}
Suppose $A$ is of bounded Tate type and that $R$ has finite homological dimension.
Then the adjunction $$\D(A) \rightleftarrows \D(B)_\cart$$ restricts to an
equivalence $\C \overset{\sim}{\longrightarrow} \D$.
\end{theorem}
\begin{proof}
The steps are as for theorem (\ref{C-D-equiv}). Our assumption on $R$ is needed
in the analogues of corollary (\ref{uniform-bound}), lemma (\ref{juhtgd}),
lemma (\ref{main-tot}), lemma (\ref{D-uniformly}),
lemma (\ref{C-D}) and lemma (\ref{D-C}). 

All other steps are exactly as for theorem (\ref{C-D-equiv}).
\end{proof}
As in the case over the integers we define subcategories
$\Perf(A) \subset \D(A)$ and $\Perf(B^\bullet) \subset \D(B^\bullet)_\cart$.

\begin{theorem} \label{RR-C-D}
Suppose $A$ is of bounded Tate type and that $R$ has finite homological dimension.
Then the natural functor
$\Perf(A) \to \Perf(B^\bullet)$ is an equivalence.
\end{theorem}
\begin{proof}
By theorem \ref{R-C-D-equiv} we are left to prove that the functor
in question is essentially surjective.
This works exactly as in the proof of proposition \ref{true-perf}.
\end{proof}

\begin{remark} \label{RR'-remark}
It is possible to generalize theorems (\ref{R-C-D}) and
(\ref{R-perf-equiv}) to the relative case where we are given a map of commutative
algebras $R \to R'$ and $R$ satisfies the assumption of theorem (\ref{RR-C-D}).
We start with an algebra $A$ of bounded Tate type over $R$ as above and
build $A_{R'}=A \otimes^\bL_R R'$ and $B_{R'}^\bullet=B^\bullet \otimes^\bL_R R'$.
Then the corresponding versions of theorems (\ref{R-C-D}) and
(\ref{R-perf-equiv}) are still valid.
\end{remark}

\section{Representations of $\bG_m$} \label{Gm-rep}

In the whole section we work with strictly commutative
algebras and the usual notion of
module (in contrast to the possibility that
we view a strictly commutative algebra as
an $E_\infty$-algebra and work with the corresponding
notion of module).

\medskip

We denote by $G$ the
group scheme $\bG_m$ over $\integers$.
Let $A=\integers[z^{\pm 1}]$ be the corresponding
Hopf algebra and $A^n$ its $n$-th tensor power as an algebra.
The $A^n$ assemble to the cosimplicial algebra $A^\bullet$
in $\Cpx(\Ab)$ modelling $G$.

We denote by $\Rep^\str(G)$ the category of strict
representations of $G$ with values in
$\Cpx(\Ab)$, i.e. the category of complexes of
$A$-comodules. Also we denote by $\Mod(A^\bullet)$ the category
of complexes of $A^\bullet$-modules.

There is an embedding $\iota^\str: \Rep^\str(G)
\to \Mod(A^\bullet)$ whose image is the
subcategory of strictly cartesian $A^\bullet$-modules.

Let $j_1: \Cpx(\Ab)^\integers \to \Rep^\str(G)$
be the functor which assigns to a complex
$M$ sitting totally in Adams degree $r$
the unique representation of $G$ on $M$ with
weight $r$ and which commutes with sums.

So the value of $j_1$ on $(M(r))_{r \in \integers}$
is $\bigoplus_{r \in \integers} M(r)\otimes \integers(r)$,
where we denote $\integers(r)$ the weight $r$ representation
of $G$ on $\integers$.

Denote by $j$ the composition $\iota^\str \circ j_1$.
Note this is a tensor functor.

Clearly we have a derived functor
$$\bL j: \D(\Ab)^\integers \to \D(A^\bullet)$$
which factors as
$$\bL j: \D(\Ab)^\integers \to \D(A^\bullet)_\cart.$$

The functor $j$ has a right
adjoint $u$ given by
$$u(F^\bullet)(r)_m=
\Hom_{\Mod(A^\bullet)}(j(D^{m,r}(\integers)),F^\bullet)
\text{.}$$

To show that $\bL j$ has a right
adjoint it is easiest to show that $j$
is a left Quillen functor for adequate model structures.
On $\Cpx(\Ab)^\integers$ we take the usual projective
model structure. The model structure
on $\Mod(A^\bullet)$ should fulfill the
property that levelwise (for the cosimplicial
direction) projective cofibrations are cofibrations,
which for example is fulfilled by the
injective model structure (i.e. a map is
a cofibration if and only if it is a monomorphism).
Hence it follows that $u$ has a right derived functor
$\bR u$ which is right adjoint to $\bL j$.
Note that for this argument we do not need to use
a symmetric monoidal model structure on $\Mod(A^\bullet)$.

We first give a proof of the following
fact.

\begin{lemma} \label{Gm-perf-rep}
The restriction of $\bL j$ to perfect objects
gives a fully faithful embedding of tensor triangulated categories
$$(\D(\Ab)^\integers)^\perf \to \Perf(A^\bullet)
\text{.}$$
\end{lemma}

Before starting the proof we give an explicit
description of the cosimplicial algebra $A^\bullet$
and of the cosimplicial $A^\bullet$-module $M^\bullet$
corresponding to a representation of pure weight $r$
on a complex $M \in \Cpx(\Ab)$.

First we fix our conventions for simplicial
and cosimplicial objects.

For $n>0$ and $0 \le i \le n$ we denote by $d_i$
the unique strictly monotone map $[n-1] \to [n]$
which omits $i$ in the target.

Now let $A$ be any commutative Hopf algebra
(over some ground ring, which we omit in the notation)
with comultiplication $\bigtriangledown: A \to A \otimes A$
and associated cosimplicial algebra $A^\bullet$.

The cosimplicial maps $d_i^A:
A^{\otimes (n-1)} \to A^{\otimes n}$
are given as
\begin{equation}
d_i^A(a_1\otimes \cdots \otimes a_{n-1})= \left\{
\begin{array}{ll}
1 \otimes a_1 \otimes \cdots \otimes a_{n-1} & i=0 \\
a_1 \otimes \cdots \otimes a_{i-1} \otimes
\bigtriangledown (a_i) \otimes \cdots \otimes a_{n-1}
& 0 < i < n \\
a_1 \otimes \cdots \otimes a_{n-1} \otimes 1 & i=n
\end{array}
\right.
\end{equation}

Now let $c: M \to M \otimes A$ be a right $A$-comodule structure
on an object $M \in \Cpx(\Ab)$.
Then we get an associated $A^\bullet$-module $M^\bullet$ as
follows: We have $M^n = M \otimes A^{\otimes n}$
and the cosimplicial map $d_i^M: M^{n-1} \to M^n$
lying over $d_i^A$ is defined as
\begin{equation}
d_i^M(m \otimes a_1\otimes \cdots \otimes a_{n-1})= \left\{
\begin{array}{ll}
c(m) \otimes a_1 \otimes \cdots \otimes a_{n-1} \otimes m
& i=0 \\
m \otimes a_1 \otimes \cdots \otimes a_{i-1} \otimes
\bigtriangledown (a_i) \otimes \cdots \otimes  a_{n-1}
& 0 < i < n \\
m \otimes a_1 \otimes \cdots \otimes  a_{n-1} \otimes 1 & i=n
\end{array}
\right.
\end{equation}

Specializing to $A=\integers[z^{\pm 1}]$
we get the following description of $A^\bullet$:

\begin{itemize}
\item $A^n=\integers[z_1^{\pm 1}, \ldots, z_n^{\pm 1}]$,
\item $d_0^A(z_l)=z_{l+1}, l=1, \ldots, n-1$,
\item for $0 < i < n$:

$d_i^A(z_l)=z_l, l\le i-1$,

$d_i^A(z_i)=z_i z_{i+1}$,

$d_i^A(z_l)=z_{l+1}, i+1 \le l \le n-1$,
\item $d_n^A(z_l)=z_l, l=1, \ldots, n-1$.
\end{itemize}

The cosimplicial module
\begin{equation} \label{cosimp-weight-r}
\integers(r)^\bullet
\end{equation}
corresponding to the weight $r$ representation $\integers(r)$
has $M^n =A^n$, $d_0^{\integers(r)}(b)=d_0^A(b) \cdot z_0^r$,
all other coface maps are the same as for $A^\bullet$.

\begin{proof}[Proof of lemma \ref{Gm-perf-rep}]
The full subcategory of perfect objects of $\D(\Ab)^\integers$ is
is generated by objects
of the form $\integers(r)$, $r \in \integers$, hence to prove fully
faithfulness of $\bL j$ on perfect objects
it is sufficient to see that the map
\begin{equation} \label{hom-perf}
\Hom_{\D(\Ab)^\integers}(\integers, \integers(r)[k])
\to \Hom_{\D(A^\bullet)}(\integers, \integers(r)[k])
\end{equation}
is an isomorphism for $r,k \in \integers$. The left hand side of (\ref{hom-perf})
is $\integers$ if $r=k=0$, otherwise it is $0$.
For fixed $r$ the right hand side is computed
by the complex associated to the cosimplicial
abelian group $\integers(r)^\bullet$, see (\ref{cosimp-weight-r}).
We describe the corresponding normalized chain
complex obtained by dividing out the images of the
$d_i^{\integers(r)}$, $0 \le i \le n-1$, in $\integers(r)^n=A^n$
and using $\partial_n:=(-1)^n d_n^{\integers(r)}: 
A^{n-1}/\{\mathrm{images}\} \to A^n/\{\mathrm{images}\}$
as differential. The differentials $d_i$ respect the
direct sum decompositions of the $A^n$ by
the monomials in the $z_i$, $i=1, \ldots, n$.
Let $n > 0$. A monomial $z_1^{e_1} \cdots z_n^{e_n} \in A^n$
lies in the image of one of the $d_i$, $i=0, \ldots, n-1$,
if and only if $e_1=r$ or $e_j=e_{j+1}$ for some
$1 \le j \le n-1$. Suppose a non-zero monomial 
$z_1^{e_1} \cdots z_n^{e_n} \in A^n/\{\mathrm{images}\}$
is mapped to $0$ by the differential $\partial_{n+1}$.
This is the case if and only if $e_n=0$, i.e.
if it is in the image of the differential
$\partial_n$. So the cohomology of the
normalized chain complex of $\integers(r)^\bullet$
is $0$ in positive degrees.
The differential $\partial_1: A^0 \to A^1/ \im(d_0)$
is $0$ if $r=0$ and injective otherwise,
hence we see that the map (\ref{hom-perf})
is indeed an isomorphism.
\end{proof}

\begin{proposition} \label{graded-full-embed}
The functor $\bL j$ is a full embedding.
\end{proposition}

\begin{proof}
By lemma \ref{Gm-perf-rep} we know that
$\bL j$ is an embedding on perfect objects.
Since $\D(\Ab)^\integers$ is generated as
triangulated category by objects
$M=\bigoplus_{i \in I} \integers(k_i)[l_i]$
we are finished if
\begin{equation} \label{comm-sum-equn}
\bigoplus_{i \in I} \Hom_{\D(A^\bullet)} (\integers,
\integers(k_i)[l_i]) \to \Hom_{\D(A^\bullet)} (\integers,
\bigoplus_{i \in I} \integers(k_i)[l_i])
\end{equation}
is an isomorphism.
We denote the $A^\bullet$-module corresponding to $M$
by $M^\bullet$.
By theorem \ref{Gm-perf-rep} the left hand side
of \ref{comm-sum-equn} is
\begin{equation} \label{hom-lhs-Gm}
\bigoplus_{i \in I, k_i=l_i=0} \integers.
\end{equation}
The right hand side of (\ref{comm-sum-equn})
is given as $H_0$ of the (derived) total object
$\Tot(M^\bullet)$ of the cosimplicial complex $M^\bullet$.
It is known that the total object is given as
the total complex of the double complex
associated to $M^\bullet$, where for the
value of the total complex in cohomological degree $n$
one has to use the {\em product} $\prod_{p+q=n} M^{p,q}$.
(One can also use the normalized associated complex.)
We denote this total complex by $\Tot(M^\bullet)$.
The only differential in this total complex comes
from the cosimplicial structure. Hence
$\Tot(M^\bullet)$ splits as
a product of complexes $\prod_{n \in \integers}
\Tot(M_n^\bullet)$, where
$M_n^\bullet$ is the $A^\bullet$-module
corresponding to $\bigoplus_{i \in I, l_i=n}
\integers(k_i)[l_i]$.
By exactness of products in $\Ab$ we get
that the right hand side of (\ref{comm-sum-equn})
is
\begin{equation} \label{hom-rhs-Gm}
\prod_{n \in \integers} H_0(\Tot(M_n^\bullet))=
H_0(\Tot(M_0^\bullet))=\bigoplus_{i \in I, k_i=l_i=0} \integers
\text{,}
\end{equation}
and it is easy to see that the map (\ref{comm-sum-equn})
is the identity on the identifications
(\ref{hom-lhs-Gm}) and (\ref{hom-rhs-Gm}).
\end{proof}

\begin{lemma} \label{ess-surj-bounded}
Every homotopy cartesian $A^\bullet$-module $M^\bullet$ such that the cohomology
of $M^0$ is bounded lies in the essential image of $\bL j$.
\end{lemma}
\begin{proof}

Without loss of generality we can assume
that $H^i(M^0)=0$ for $i<0$ and if
$M^0$ is non-contractible that $H^0(M^0) \neq 0$.
Let $l$ be the largest integer such that $H^l(M^0)
\neq 0$ if $M^0$ is non-contractible, otherwise
we set $l=-1$.
We prove the statement by induction on $l$.

The start of the induction is $l=-1$ where we have nothing
to prove. So $l \ge 0$.
By replacing $M^{n,0}$ by
$M^{n,0}/\mathrm{image}(M^{n,-1})$
and setting $M^{n,i}=0$ for $i<0$
we can assume that the complexes $M^n$ sit purely
in cohomologically non-negative degrees.
The cycles $Z^n=\ker(M^{n,0} \to M^{n,1})$
equal the cohomology $H^0(M^n)$ and
form a strictly cartesian $A^\bullet$-module.
They thus form a direct sum
of representations of $\bG_m$ of the form
$M_r(r)$ for abelian groups $M_r$. Hence $Z^\bullet$ lies in the
image of $\bL j$. Applying the induction hypothesis
to the (shifted) cofiber of $Z^\bullet \to M^\bullet$
we see that $M^\bullet$ itself is a cofiber
of objects from the image of $\bL j$.
Using fully faithfulness (proposition (\ref{graded-full-embed}))
yields the induction step.
\end{proof}

\begin{proposition} \label{ess-surj}
The functor $\bL j \colon \D(\Ab)^\integers \to \D(A^\bullet)_\cart$ is essentially surjective.
\end{proposition}
\begin{proof}
Let $M^\bullet$ be a homotopy cartesian $A^\bullet$-module. We will apply
levelwise good truncation functors to $M^\bullet$. So for $S \in \{\le a, \ge a, [a,b]\}$
denote by $\tau^S(M^\bullet)$ the good truncation such that the cohomology is preserved
in the indicated region and is $0$ otherwise. We have an exact triangle
$$\tau^{\le 0}(M^\bullet) \to M^\bullet \to \tau^{\ge 1}(M^\bullet) \to \tau^{\le 0}(M^\bullet)[1]$$
in $\D(A^\bullet)_\cart$. To prove that $M^\bullet$ is in the essential image of
$\bL j$ it is thus sufficient by proposition (\ref{graded-full-embed}) that
$\tau^{\le 0}(M^\bullet)$ and $\tau^{\ge 1}(M^\bullet)$ are in the essential image of $\bL j$.

Let us do first the case of $\tau^{\le 0}(M^\bullet)$:
We can write $$\tau^{\le 0}(M^\bullet) \simeq \holim_{a \to - \infty} \tau^{[a,0]}(M^\bullet).$$
Now the functor $\bR u \colon \D(A^\bullet) \to \D(\Ab)^\integers$ preserves homotopy limits,
so we have 
$$\bR u(\tau^{\le 0}(M^\bullet)) \simeq \holim_{a \to - \infty} \bR u(\tau^{[a,0]}(M^\bullet)).$$
By lemma (\ref{ess-surj-bounded}) the $i$-th cohomology of $\bR u(\tau^{\le 0}(M^\bullet))$ is
thus given by $H^i(M^\bullet)$ for $i \le 0$ (viewing the $\mathbb{G}_m$-representation $H^i(M^\bullet)$
as a graded object).
The commutativity of the diagram
$$\xymatrix{\bL j \bR u (\tau^{\le 0}(M^\bullet)) \ar[r] \ar[d] & \tau^{\le 0}(M^\bullet) \ar[d] \\
\bL j \bR u (\tau^{[a,0]}(M^\bullet)) \ar[r] & \tau^{[a,0]}(M^\bullet)}$$
now shows that the upper horizontal map is an equivalence.

We turn to showing that $\tau^{\ge 1}(M^\bullet)$ lies in the essential image of $\bL j$.
We can write $$\tau^{\ge 1}(M^\bullet) \simeq \hocolim_{a \to \infty} \tau^{[1,a]}(M^\bullet).$$ 
Let $X_a := \bR u (\tau^{[1,a]}(M^\bullet))$ and set $X := \hocolim_{a \to \infty} X_a$.
Then $$\bL j (X)=\hocolim_{a \to \infty} \bL j (X_a)$$ and we have a canonical map
$\bL j (X) \to \tau^{\ge 1}(M^\bullet)$. By lemma (\ref{ess-surj-bounded}) this is
an isomorphism on cohomology hence an equivalence. This finishes the proof of
the essential surjectivity of $\bL j$.
\end{proof}
\begin{theorem} \label{j-equiv}
The functor $\bL j \colon \D(\Ab)^\integers \to \D(A^\bullet)_\cart$ is an equivalence
of tensor triangulated categories.
\end{theorem}
\begin{proof}
This combines propositions (\ref{graded-full-embed}) and (\ref{ess-surj}).
\end{proof}

\begin{corollary} \label{j-R-equiv}
Let $R$ be a commutative $\S$-algebra in $\Cpx(\Ab)^\integers$. Let $R'$ be its
image in $\Mod(A^\bullet)$ under $j$. Then the induced functor
$\D(R) \to \D(R')_\cart$ is an equivalence
of tensor triangulated categories.
\end{corollary}
\begin{proof}
This follows from theorem (\ref{j-equiv}) and the fact that the
adjoint functors $\D(R) \rightleftarrows \D(R')_\cart$ commute with forgetting
the $R$- resp. $R'$-module structure.
\end{proof}

\section{Transfer argument} \label{transfer}

We keep the notation from section (\ref{Gm-rep}),
so $A^\bullet$ denotes the cosimplicial algebra
corresponding to $\bG_m$.
We consider a cosimplicial commutative $\S$-algebra $B^\bullet$
in $\Cpx(\Ab)^\integers$ and denote $C^\bullet=j(B^\bullet)$,
i.e. $C^\bullet$ is a cosimplicial algebra in $\Mod(A^\bullet)$.
In particular $C^\bullet$ is a bicosimplicial commutative $\S$-algebra in $\Cpx(\Ab)$.

There is a functor $j_B: \Mod(B^\bullet) \to \Mod(C^\bullet)$
which has a left derived functor $\bL j_B$.
Both $j_B$ and $\bL j_B$ are tensor functors.

As $j$ the functor $j_B$ has a right adjoint,
which we denote by $u_B$.
Using model structure one sees again that
$u_B$ has a right derived functor
$\bR u_B$.

We denote by $\D(C^\bullet)_\cart$ the homotopy category of $C^\bullet$-modules
which are homotopy cartesian as bicosimplicial modules.

\begin{corollary} \label{transfer-corollary}
The natural functor $\bL j_B \colon \D(B^\bullet)_\cart \to \D(C^\bullet)_\cart$ is an equivalence
of tensor triangulated categories.
\end{corollary}

\begin{proof}
This follows from corollary (\ref{j-R-equiv}) and the fact that the adjoint functors
in question commute with the functors restricting a cosimplicial module to one level.
\end{proof}
\medskip

Recall the maps $\alpha_i \colon [1] \to [n]$ defined before definition
(\ref{defi-derived-affine-gr-scheme}).

\begin{lemma} \label{bicos-help}
Let $n>0$. Let $R^\bullet$ be an affine derived group scheme in $\Cpx(\Ab)$ over $\integers$ and let
$S^\bullet$ be a homotopy cartesian commutative $R^\bullet$-$\S$-algebra. Then the
map $(S^1)^{\otimes^\bL n} \to (S^n)^{\otimes^\bL_{R^n} n}$ induced by the maps $\alpha_0, \ldots,\alpha_{n-1}$
is an equivalence.
\end{lemma}
\begin{proof}
We let $R'=(R^1)^{\otimes^\bL n}$ and $\beta_i \colon R^1 \to R'$ the $i$-th inclusion,
$1 \le i \le n$. Let $S_i$ be the push forward of $S^1$ with respect to $\beta_i$.
Then the map $$(S^1)^{\otimes^\bL n} \to S_1 \otimes^\bL_{R'} \cdots \otimes^\bL_{R'} S_n$$
induced by the natural maps $S^1 \to S_i$ is an equivalence.

There are maps $S_i \to S^n$ induced by the maps $(\alpha_{i-1})_* \colon S^1 \to S^n$
lying over the map $(R^1)^{\otimes^\bL n} \to R^n$ induced by the $\alpha_0, \ldots, \alpha_{n-1}$.

The latter map is an equivalence since $R^\bullet$ is an affine derived group scheme.
Thus the induced map $$S_1 \otimes^\bL_{R'} \cdots \otimes^\bL_{R'} S_n \to
(S^n)^{\otimes^\bL_{R^n} n}$$ is also an equivalence. This shows the claim.
\end{proof}

We keep the notations from this paragraph.
We assume now that $B^\bullet$ represents
an affine derived group scheme over $\unit$, i.e. fulfills
the conditions of
definition \ref{defi-derived-affine-gr-scheme}.
Then $C^\bullet$ also represents an affine derived group scheme in $\Mod(A^\bullet)$.

We denote by $D^\bullet$ the diagonal of the bicosimplicial commutative $\S$-algebra
underlying $C^\bullet$.

\begin{lemma}
The cosimplicial commutative $\S$-algebra $D^\bullet$ in
$\Cpx(\Ab)$ is an affine derived group scheme over $\integers$.
\end{lemma}

\begin{proof}
We denote by $C^{n,k}$ the entry in cosimplicial degree $k$ of the $A^\bullet$-algebra
$C^n$. Since $B^\bullet$ is an affine derived group scheme we know
that the map $$C^{1,k} \otimes^\bL_{A^k} \cdots \otimes^\bL_{A^k} C^{1,k} \to C^{k,k}$$
induced by the maps $\alpha_0, \ldots, \alpha_{k-1}$ in the first cosimplicial direction
is an equivalence.

Furthermore by lemma (\ref{bicos-help}) the map
$$(C^{1,1})^{\otimes^\bL k} \to (C^{1,k})^{\otimes^\bL_{A^k} k}$$
induced by the $\alpha_i$ in the second cosimplicial direction
is an equivalence. This establishes that also the composite map
$$(C^{1,1})^{\otimes^\bL k} \to C^{k,k}$$ induced by the $\alpha_i$ in the
diagonal cosimplicial direction is an equivalence.
This shows the first property necessary for being an affine derived group scheme.

The proof for the second condition for being an affine derived group
scheme is similar:

As in lemma (\ref{bicos-help}) it follows that the map
$$C^{1,1} \otimes^\bL C^{1,1} \to C^{1,2} \otimes^\bL_{A_2} C^{1,2}$$
induced by the maps $c,\alpha_0$ in the second cosimplicial direction
is an equivalence since $A^\bullet$ satisfies the second condition
of being an affine derived group scheme
and since $C^{1,\bullet}$ is homotopy cartesian over $A^\bullet$.

Furthermore the map
$$C^{1,2} \otimes^\bL_{A_2} C^{1,2} \to C^{2,2}$$
induced by  $c,\alpha_0$ in the first cosimplicial direction
is an equivalence since $B^\bullet$ satisfies the second condition
of being an affine derived group scheme.

Again the composite of these two maps yields the map which the second
condition of being an affine derived group scheme requests to be
an equivalence.

The third condition that the map $\integers \to D^0$ is an equvalence is satisfied
since $A^0=\integers$ and the map $\unit \to C^0$ is an equivalence.

This finishes the proof of the claim.
\end{proof}

Morally we can think of $D^\bullet$ as the derived affine
group scheme
\begin{equation} \label{semi-direct}
B^\bullet \rtimes \mathbb{G}_m.
\end{equation}

We will use this notation in the section on examples.

\begin{proposition} \label{proposition:diag-equiv}
The restriction functor $\D(C^\bullet)_\cart \to \D(D^\bullet)_\cart$
is an equivalence of categories.
\end{proposition}

\begin{proof}
This follows form corollary (\ref{corollary:diag-equiv}).
\end{proof}

\begin{theorem} \label{main-transfer-thm}
There is a natural equivalence of tensor triangulated categories
$\D(B^\bullet)_\cart \to \D(D^\bullet)_\cart$. It restricts to an equivalence
$\Perf(B^\bullet) \to \Perf(D^\bullet)$.
\end{theorem}

\begin{proof}
This combines corollary (\ref{transfer-corollary}) and proposition (\ref{proposition:diag-equiv}).
\end{proof}

We are going to describe the above result in the case of algebras
over a given commutative ring $R$. 

Set $A_R^\bullet=A^\bullet \otimes R$. Let $B^\bullet$ be an affine derived group scheme
in $\Cpx(R)^\integers$. Let $C^\bullet=j(B^\bullet)$. So $C^\bullet$ is a derived affine
group scheme in $\Mod(A_R^\bullet)$. In particular it is a bicosimplicial commutative
$\S$-algebra in $\Cpx(R)$. Let $D^\bullet$ be its diagonal.

\begin{theorem} \label{R-main-transfer-thm}
There is a natural equivalence of tensor triangulated categories
$\D(B^\bullet)_\cart \to \D(D^\bullet)_\cart$. It restricts to an equivalence
$\Perf(B^\bullet) \to \Perf(D^\bullet)$.
\end{theorem}

\begin{proof}
This combines the $R$-analogues of
corollary (\ref{transfer-corollary}) and proposition (\ref{proposition:diag-equiv}).
\end{proof}

\begin{proof}[Proof of theorem (\ref{main-thm})]
If the characteristic of $k$ is $0$
then by \cite[Corollary 6.9]{spitzweck.per} there exists a commutative $\S$-algebra
$A$ in $\Cpx(\Ab)^\integers$ such that $\D(A)$ is naturally equivalent
to $\DMT(X)$ as tensor triangulated category.
In general this follows from the fact that over any scheme which lives
over a field the Eilenberg MacLane spectrum $\MZ$ is strongly
periodizable in the language of \cite{spitzweck.per}. This follows
from \cite[Corollary 6.3]{spitzweck.per} for perfect fields, in general
one uses the fact that there is a map of $E_\infty$-ring spectra
from the pullback of the Eilenberg MacLane spectrum over the prime field
to the Eilenberg MacLane spectrum over the given base.

Now $\MZ$ receives a canonical map from
the push forward of the topological Eilenberg MacLane spectrum. This
implies that there is an $\S$-algebra 
$A$ in $\Cpx(\Ab)^\integers$ such that $\D(A)$ is naturally equivalent
to $\DMT(X)$ as tensor triangulated category by a representation
theorem similar to \cite[Corollary 6.9]{spitzweck.per}.

The equivalence $\DMT(X) \simeq \D(A)$ restricts to an equivalence $\DMT_\gm(X) \simeq \Perf(A)$.
Note that $A$ is of bounded Tate type by our assumptions on $\DMT(X)$. 
We replace $A$ by the canonical algebra of strict Tate type and denote
it again by $A$. Factor the canonical augmentation
$A \to \unit$ into a cofibration $A \to B$ followed by a weak equvalence
$B \to \unit$. Let $A \to B^\bullet$ the coaugmented cosimplicial algebra assciated
to $A \to B$. Then by proposition (\ref{cech-der-equiv}) $B^\bullet$ is
an affine derived group scheme over $\unit$. Theorem (\ref{perf-descent-thm})
gives us an equivalence of tensor triangulated categories
$\Perf(A) \simeq \Perf(B^\bullet)$.

We next apply the procedure of this paragraph. So let $C^\bullet=j(B^\bullet)$
and let $D^\bullet$ be the diagonal of $C^\bullet$. Then theorem
(\ref{main-transfer-thm}) gives an equivalence of tensor
triangulated categories $\Perf(B^\bullet) \simeq \Perf(D^\bullet)$.
So $D^\bullet$ is our looked for affine derived group scheme with the property
$\DMT_\gm(X) \simeq \Perf(D^\bullet)$.

For the case with $R$-coefficients observe first that there is an
equivalence $$\DMT_\gm(X)_R \simeq \Perf(A_R).$$ Next theorem (\ref{R-perf-equiv})
gives the equivalence $\Perf(A_R) \simeq \Perf(B_R^\bullet)$. Theorem
(\ref{R-main-transfer-thm}) yields $\Perf(B_R^\bullet) \simeq \Perf(D_R^\bullet)$
which gives the conclusion.
\end{proof}

\begin{proof}[Proof of theorem (\ref{R-main-thm})]
The proof goes along the same lines as the proof of theorem (\ref{main-thm})
using remark (\ref{RR'-remark}).
\end{proof}

\begin{proof}[Proof of theorem (\ref{main1-thm})]
As in the proof of theorem (\ref{main-thm}) we can find a Tate-algebra
$A$ in $\Cpx(\Q)^\integers$ such that $\DMT(X)_\Q \simeq \D(A)$
and $\DMT_\gm(X)_\Q \simeq \Perf(A)$. Now we claim that we can write
$A$ as a filtered (homotopy) colimit of {\em bounded} Tate-algebras,
$A \simeq \colim_i A_i$. That granted we apply the procedure used 
in the proof of theorem (\ref{main-thm}) to get affine derived group schemes
$D_i^\bullet$ such that $\Perf(A_i) \simeq \Perf(D_i^\bullet)$.
The $D_i^\bullet$ can be defined in such a way that they depend functorially
on $i$. In this way we get a pro affine derived group scheme
$\text{``$\lim_i$''} D_i^\bullet$. Noting that $\Perf(A) \simeq \text{$2$-$\colim$}_i \Perf(A_i)$
we get equivalences $$\Perf(\text{``$\lim_i$''} D_i^\bullet)=\text{$2$-$\colim$}_i \Perf(D_i^\bullet)$$
$$\simeq \text{$2$-$\colim$}_i \Perf(A_i) \simeq \Perf(A) \simeq \DMT_\gm(X)_\Q$$
which was to be shown.

It remains to prove that we can write $A$ as a filtered (homotopy) colimit of
bounded Tate algebras. We give two procedures. The first one (mentioned in
the introduction) writes $A$ as a filtered (homotopy) colimit of finite type cell
Tate-algebras. So we have to prove that such a a finitely generated Tate-algebra is bounded.
Here we use that we work over $\Q$. In this case we can work with strictly commutative
(Adams graded) DGA's. We approximate $A$ by finite cell algebras where the cells are
attached in negative Adams degrees. It is easily seen that such an algebra is of bounded
Tate type (forgetting the differential it is free on finitely many generators in negative Adams degree). 

The second way to approximate $A$ (which we can assume to be of strict Tate type)
is to truncate $A$ in the following ways: For $i \in \naturals_{> 0}$, $k<0$,$n \in \integers$ we set
$A_i(k)^n=0$ if $n < ik$ and $A_i(k)^n=A(k)^n$ otherwise. Then the $A_i$ form subalgebras
of $A$ such that $A=\colim_i A_i$, and each $A_i$ is bounded.

This finishes the proof.
\end{proof}

\begin{remark}
The two strategies employed in the above proof to write a Tate-algebra as a filtered (homotopy)
colimit of bounded Tate-algebras does not work for Tate algebras over the integers.
For the first method we note that a finite cell algebra with generators in negative
Adams degrees will in general not be bounded because of the contributions of the homology
of the symmetric groups. The second method fails since the truncations considered will
in general not be $\S$-algebras since the involved $E_\infty$-operad lives in
negative cohomological degrees.
\end{remark}

 \section{Examples} \label{section:examples}

Our general theorem (\ref{main1-thm}) applies to all smooth schemes over fields
with rational coefficients. In its form it deals with representation categories
of pro affine derived group schemes.
In this section we focus on examples of theorems (\ref{main-thm}) and (\ref{R-main-thm})
which discuss situations where Tate motives can be modelled as representation
categories of affine derived group schemes.

We assume the reader is familiar with the constructions of
affine derived group schemes starting with an Adams graded commutative $\S$-algebra
in complexes with an augmentation.
Below we will always apply theorems (\ref{perf-descent-thm}), (\ref{R-perf-equiv}),
(\ref{RR-C-D}), remark (\ref{RR'-remark}) and theorems (\ref{main-transfer-thm}),
(\ref{R-main-transfer-thm}).

\subsection{First examples}

Let $k$ be a field of characteristic $0$ and
$A$ the algebra in $\Cpx(k)^\integers$ where $A(0)=S^0(k)$,
$A(-1)=S^{-1}(k \oplus k)$, and all other complexes are equal to $0$. 
That is the typical cohomology algebra of $\mathbf{CP}^1 \setminus \{0,1,\infty\}$
where we put the generators for $H^1$ in Adams degree $-1$.

The algebra $B^1=k \otimes_A^\bL k$ sits completely in cohomological
degree $0$, and the Hopf algebra structure determines the pro-unipotent
group scheme with $\mathbb{G}_m$-action which corresponds to the completed
free Lie algebra over $k$ on two generators in Adams degree $1$.

If we put $A=k[x]/(x^{n+1})$, where $x$ sits in cohomological degree $2$ and Adams
degree $-1$, then $A$ recovers the $k$-cohomology of $\mathbf{CP}^n$
where we put the generator in cohomological degree $2$ into Adams degree $-1$.
In this case the algebra $B^1=k \otimes_A^\bL k$ sits in various cohomological
degrees.

In the case $A=k[x]$ with the bidegree of $x$ as above we have that $B^1=k \otimes_A^\bL k$
has a single generator in cohomological degree $1$ and Adams degree $-1$.

In these cases our representation theorem applies
and we have $$\Perf(A)= \Perf(B^\bullet \rtimes \mathbb{G}_m).$$
Here $B^\bullet \rtimes \mathbb{G}_m$ is to be understood
as in notation (\ref{semi-direct}) from section (\ref{transfer}).

Since these algebras are base changes from algebras over $\unit \in \Cpx(\Ab)^\integers$
our representation theorems apply to their versions over an arbitrary coefficient ring.
In particular we have their versions with $\integers$-coefficients.

\subsection{Finite fields}

Next we give examples for base schemes and coefficients where our theorems
(\ref{main-thm}) and (\ref{R-main-thm}) apply.

First let $k$ be a finite field with $q=p^e$ elements. Let $R=\integers[\frac{1}{p}]$.
Using the Bloch-Kato conjecture Levine computed the motivic cohomology of $k$
(\cite[Remark 14.11]{levine.mot-coh}):

$$\begin{array}{lcl}
\Hom_{\DM(k)_R}(R(0),R(0)) & = & R, \\
\Hom_{\DM(k)_R}(R(0),R(n)[1]) & = & R/(q^n-1)R, \;\; n\ge 1,
\end{array}$$

and all other bi-homs are equal to $0$.
Thus theorem (\ref{R-main-thm}) applies and we have a derived motivic
fundamental group $B^\bullet$ over $R$ such that $\DMT_\gm(k)_R \simeq \Perf(B^\bullet)$.

It is tempting to compute $B^1$ in this simple case.
However it is unclear what the multiplication $A \otimes^\bL_R A \to A$,
$A$ the algebra of bounded Tate type modelling the Tate motives, does
on the Tor-terms.

\subsection{Number fields}

Next let $k$ be a number field. Levine's computation (\cite[Remark 14.11]{levine.mot-coh})
in this case gives an isomorphism
$$\Hom_{\DM(k)}(\integers_{(l)}(0), \integers_{(l)}(n)[i]) \otimes_{\integers_{(l)}}
\integers_l \cong H_{\acute{e}t}^i(k,\integers_l(n))$$
for all $i \le n$.

Thus theorem (\ref{main-thm}) applies and we have a derived motivic
fundamental group $B^\bullet$ over $\integers$ such that $$\DMT_\gm(k) \simeq \Perf(B^\bullet).$$
This fundamental group is the promised integral structure on the usual (non-derived)
fundamental group for number fields.

\subsection{Finite coefficients}

We look now at the special case where the coefficients are an algebra over a
finite field, so we suppose that $R'$ is an $R=\mathbb{F}_p$-algebra, $p$ a prime.
Suppose $k$ is a field of characteristic not equal to $p$.

Then by the Bloch-Kato conjecture we have $$\Hom_{\DM(k)_R}(R(0),R(n)[i]) \cong
H_{\acute{e}t}^i(k,R(n)),$$ thus in particular the assumptions of
theorem (\ref{R-main-thm}) are fulfilled. So we have  derived motivic
fundamental group $B^\bullet$ over $R'$ such that
$$\DMT_\gm(k)_{R'} \simeq \Perf(B^\bullet).$$

\subsection{Geometric fundamental groups}

Let $R$ be an algebra fulfilling the assumptions of theorem (\ref{R-main-thm}).
Suppose $k$ gives rise to an algebra $A$ over $R$ of bounded Tate type,
i.e. that the assumptions of theorem (\ref{R-main-thm}) are fulfilled as e.g. in all previous examples.
Let $X=\PP^1_k \setminus \{0,1,\infty\}$. We denote the corresponding
algebra modelling Tate motives over $X$ by $A(X)$. Let $B^\bullet$ and
$B^\bullet(X)$ be the derived affine schemes corresponding to
$A$ and $A(X)$. Then the natural map of cosimplicial algebras
$B^\bullet \to B^\bullet(X)$ can be thought of as a map of group schemes
in the opposite direction. The kernel of this group scheme map is an affine derived group scheme
sitting completely in degree $0$. It coincides as a group scheme
with $\mathbb{G}_m$-action with our first example in this section
defined over $R$. We think of this kernel as the geometric
fundamental group of $X$.

In a similar way we can define geometric fundamental groups of other varieties
over $k$ such as $\PP^n_k$ or Grassmannians which admit stratifications
where the strata are linear varieties.

\subsection{Beilinson motives}
\label{hgfddf}

Let $\LQ$ be the Landweber spectrum over a base scheme $S$ modelled on the rationals with its
canonical $E_\infty$-structure, see \cite[section 10]{NSO1}.
Let $\DMT_{\B,\gm}(S)$ be the full triangulated subcategory of the homotopy
category of highly structured $\LQ$-modules spanned by the motivic spheres
(note that the latter homotopy category is the category of Beilinson motives
$\DM_\B(S)$ in the sense of \cite{cisinski-deglise}).
Since $\LQ$ is rational it is canonically strongly periodizable in
the sense of \cite{spitzweck.per}. Thus the representation theorem
of loc. cit. applies and we find a commutative algebra $A$ in
$\Cpx(\Q)^\integers$ such that $\DMT_{\B,\gm}(S) \simeq \Perf(A)$.
 If we suppose that $A$ is of bounded Tate type then we get
an affine derived group scheme $B^\bullet$ over $\Q$ 
such that $$\DMT_{\B,\gm}(S) \simeq \Perf(B^\bullet).$$
In general, if $A$ is of Tate type, we get a pro affine
derived group scheme $\text{``$\lim_i$''} B_i^\bullet$ such that
$$\DMT_{\B,\gm}(S) \simeq \Perf(\text{``$\lim_i$''} B_i^\bullet).$$

The first case applies for example when the base is the spectrum of $S$-integers
of a number field $\caO_S$ or $\PP^1_{\caO_S} \setminus \{0,1,\infty\}$.

In both cases the derived fundamental group sits completely in degree $0$.
In the case of the $S$-integers we recover the motivic fundamental group defined in
\cite{deligne-goncharov}. The kernel of the induced map of fundamental
groups is the geometric fundamental group defined in loc. cit. after forgetting the
action of the motivic fundamental group of $\caO_S$. To construct this action
we would have to take care of a fiber functor induced by a (tangential) base
point of $\PP^1_{\caO_S} \setminus \{0,1,\infty\}$ in our derived setting.
We did not pursue that point further in this text.

\noindent

\bibliographystyle{plain}
\bibliography{der-fundgr}

\begin{center}
Department of Mathematics, University of Oslo, Norway.

e-mail: markussp@math.uio.no
\end{center}

\end{document}